\documentclass[a4paper,10pt]{amsart}
\pdfoutput=1
\usepackage[utf8]{inputenc}
\usepackage[english]{babel}

\usepackage[bitstream-charter]{mathdesign}

\usepackage[T1]{fontenc}

\usepackage{amsmath, amsthm} 

\usepackage{mathtools}

\usepackage{hyperref}
\usepackage{xfrac}

\newcommand{\Z}{{\mathbb{Z}}}
\newcommand{\C}{{\mathbb{C}}}
\newcommand{\R}{{\mathbb{R}}}
\newcommand{\N}{{\mathbb{N}}}
\renewcommand{\P}{{\mathbb{P}}}
\newcommand{\E}{{\mathbb{E}}}

\newcommand{\cO}{{\mathcal{O}}}
\newcommand{\cH}{{\mathcal{H}}}
\newcommand{\cA}{{\mathcal{A}}}

\newcommand{\cF}{{\mathcal{F}}}
\newcommand{\cL}{{\mathcal{L}}}
\newcommand{\cC}{{\mathcal{C}}}

\newcommand{\HS}{\text{\textsc{H\!S}}}
\renewcommand{\d}{\,\mathrm{d}}

\theoremstyle{plain}
\newtheorem{thm}{Theorem}[section]
\newtheorem{lemma}[thm]{Lemma}

\theoremstyle{definition} 
\newtheorem{defi}{Definition}
\newtheorem{assu}{Assumption}

\newtheorem{es}{Example}
\newtheorem{remark}{Remark}
\title[Modulation Equation for SPDE in unbounded domains]{Modulation Equation for SPDEs in unbounded domains with space-time white noise -- Linear Theory}

\author[L.~A.~Bianchi]{Luigi Amedeo Bianchi}
    \address{
    L.~A.~Bianchi, 
    Institut f\"ur Mathematik\\ Universit\"at Augsburg\\ D-86135 Augsburg, Germany}
    \email{\href{mailto:luigi.bianchi@math.uni-augsburg.de}{luigi.bianchi@math.uni-augsburg.de}}
    \urladdr{\url{http://www.math.uni-augsburg.de/prof/ana/arbeitsgruppe/bianchi/}}
    
    \author[D.~Bl\"omker]{Dirk Bl\"omker}
    \address{D.~Bl\"omker, 
    Institut f\"ur Mathematik\\ Universit\"at Augsburg\\ D-86135 Augsburg, Germany}
    \email{\href{mailto:dirk.bloemker@math.uni-augsburg.de}{dirk.bloemker@math.uni-augsburg.de}}
    \urladdr{\url{http://www.math.uni-augsburg.de/prof/ana/arbeitsgruppe/bloemker/}}
  \thanks{Both authors are supported by the German Science Foundation (DFG),
grant number BL 535/9-2
"Mehrskalenanalyse stochastischer partieller Differentialgleichungen
(SPDEs)"}

\date{\today}
\begin{document}

  \begin{abstract}
    We study the approximation 
   of SPDEs on the whole real line near a change of stability 
   via modulation or amplitude equations, which acts as a replacement for the 
   lack of random invariant manifolds on extended domains. 
   Due to the unboundedness of  the underlying 
   domain a whole band of infinitely many eigenfunctions changes stability.
   Thus we expect not only a slow motion in time, but 
   also a slow spatial modulation of the dominant modes, 
   which is described by the modulation equation.
   
   As a first step towards a full theory of modulation equations for 
   nonlinear SPDEs on unbounded domains,  we focus, in the results presented here,
   on the linear theory for one particular example, the Swift-Hohenberg equation.
   These linear results are one of the key technical tools to carry over 
   the deterministic approximation results to the stochastic case with additive forcing.
   One technical problem for establishing error estimates rises from 
   the spatially translation invariant nature of space-time white noise on unbounded
   domains, which implies that at any time we can expect 
   the error to be always very large somewhere in space.
  \end{abstract}

\maketitle

\section{Introduction}

We study the approximation of stochastic partial differential equations (SPDEs)
on unbounded domains near a change of stability of a trivial solution via modulation or amplitude equations. 
Due to the unboundedness of the underlying domain a whole infinite band (i.e., an interval) 
of eigenfunctions changes sign and therefore the trivial solution its stability.
Thus neither the classical theory of invariant manifolds for PDEs nor
the recently developed theory of random invariant manifolds \cite{DLS:03,MZZ:08,DuWa:14,CLW:15a,CLW:15b}
can be applied. 

Modulation or amplitude equations are a replacement to overcome the lack of invariant manifolds,
and they serve as a universal normal form depending only on the type of bifurcation.
Being widely used in the physics literature, they are a tool to describe  
the evolution of the amplitude of the dominating pattern changing stability,
where close to bifurcation we expect not only a slow motion of the amplitude in time, 
but  also a slow modulation in space due to the band of eigenvalues changing sign.

For deterministic PDEs this theory is a well-established tool. 
See for example \cite{CoEc:90,KMS92,GS96, MiSc:95} for classical references, 
and the detailed comments later in this section.
But hardly anything is known for SPDEs on unbounded domains. 

As a starting point in this paper we consider the stochastic Swift-Hohenberg equation \cite{CrHo93,HoSw92}, 
which is a reduced model for the 
first convective instability in the Rayleigh-B\'enard model
and serves as one of the main examples in which pattern formation is studied. 
It is given as
\begin{equation}
\label{e:SH}
   \frac{\partial u}{\partial t} 
   = -(1+\partial_x^2)^2u + \varepsilon^2\nu u -u^3 +\varepsilon^{3/2}\xi
\end{equation}
on the whole real line with space-time white noise $\xi$. 
As we want to allow for periodic patterns, we do not assume any decay condition of solutions at infinity. 

The theory of higher order parabolic stochastic partial differential equations (SPDEs) 
on unbounded domains with additive translation invariant noise 
like space-time white noise is not that well studied, 
while for the wave equation with multiplicative noise 
there are many recent publications (see for example \cite{Kho14, DaSS09, Dal09, FI15})
and even more recent ones for parabolic equations with very rough noise \cite{HaLa:15a,HaLa:15b}.

In many cases parabolic equations with noise are studied subject to a spatial cut off or a decay condition at infinity.
This is the case, for example, in~\cite{EckHai2001}, where the cut-off is both  in the  real space as well as in the Fourier space. Another example is~\cite{Fun1995}. 
In~\cite{BreLi2006} and in a similar way in~\cite{KruSta2014,La2015}, the authors consider $L^2$-valued solutions, where an integral equation is consider instead of a PDE.
The choice of trace class noise in these examples implies that we have an $L^2$-valued Wiener processes 
and thus a decay condition at infinity, which in both cases leads to more regular solutions.

If we were to consider decay at infinity, 
we conjecture we'd recover similar results but with 
a point-forcing in the amplitude equation, due to the rescaling in space, 
needed to obtain the modulation equation.

The scaling of the equation involves small noise of order $\cO(\varepsilon^{3/2})$ 
and small distance from bifurcation of order $\cO(\varepsilon^{2}\nu)$. Due to the closeness to bifurcation, 
we expect small solutions and slow dynamics in time. 
Moreover, a whole band of Fourier modes around wave-number $k\pm1$ changes sign close to $\mu=0$, and 
thus we expect the dynamics to be given by a slow modulation of the complex amplitude $A$ 
of the  dominant pattern $e^{\pm ix}$:
\[
u(t,x) \approx \varepsilon A(\varepsilon^2 t, \varepsilon x) \cdot e^{ix}  + c.c. \;, 
\]
where $c.c.$ denotes the complex conjugate of the previous term. 
We expect such an estimate to hold on the slow time-scale with $t=\cO(\varepsilon^{-2})$.

The noise is chosen in a way that in the limit $\varepsilon\to 0$ both
noise and linear instability do influence the dynamics of the amplitude equation.
If we scale differently, we would lose one of the effects.
The choice of space-time white 
noise is mainly for simplicity of the analysis, in order to avoid further technical difficulties, 
as we expect space-time white noise to appear
in the amplitude equation in many  cases of coloured and thus smoother noise, thanks to the scaling limit. 
For a detailed discussion on coloured noise see \cite{BlHaPa07}, 
where large but still bounded domains of order $\cO(1)$ were treated.
Moreover, on bounded domains \cite{DBMPS:01}, with fractional noise as in~\cite{Blo:05} 
or $\alpha$-stable noise, the scaling of the noise's strength is different, but the result itself is similar.


\subsection{Previous results}

In the equation without noise ($\xi=0$)  Mielke, Schneider \& Kirrmann~\cite{KMS92} showed 
(see also~\cite{GS96} or numerous other publications by the authors) 
\begin{equation}
\label{e:AEdet}
\partial_{T}A=4\partial_X^2 A+\nu A-3|A|^2 A.
\end{equation}%
For the stochastic equation~\eqref{e:SH} on large but bounded domains of size $\cO(\varepsilon^{-1})$ 
Bl\"omker, Hairer \& Pavliotis~\cite{BlHaPa07} derived the stochastic amplitude equation
\begin{equation}
\label{e:AE}
\partial_{T}A=4\partial_X^2 A+\nu A-3|A|^2 A+\eta,
\end{equation}%
on a bounded domain of order $\cO(1)$ 
with complex-valued space-time white noise, although $\xi$ could have been quite regular in space.
The idea of splitting the solution in a Gaussian and a more regular part, 
which we will use in our approximation result, was already 
present in this paper,
but due to boundedness of the domain there were no problems with growth at infinity.  
See also  Mielke, Schneider, \& Ziegra~\cite{MSZ00} for large domains and no noise.

\begin{remark}
On bounded domains the noise has to be of strength $O(\varepsilon^2)$
to get interesting results.
With that scaling  Bl\"omker, Maier-Paape, \& Schneider~\cite {DBMPS:01}
showed that the amplitude of the dominant mode  is  independent of space and derived a stochastic 
ordinary differential equation (SDE) in $\mathbb{C}$ given by
\begin{equation*}
\partial_{T}A = \nu A-3|A|^2 A+\dot\beta,
\end{equation*}%
where $\dot\beta$ is a complex-valued white noise in time.
\end{remark}

 \begin{remark}
Spatially constant noise does not act directly on the dominant modes, 
and thus for noise of order $\cO(\varepsilon^{3/2})$ the noise term would just
disappear in the amplitude equation,
and we'd only recover the deterministic one stated in~\eqref{e:AEdet}.

If we increase the noise strength to be of order $\cO(\varepsilon)$ and set the noise to be spatially independent as $\xi=\sigma\dot\beta$,
where $\dot\beta$ is the derivative of a Brownian motion in $\R$, we obtain a time only white noise, and in that case 
additional terms in the amplitude equation arise
due to nonlinear interaction of the noise with itself in Fourier space.
 Mohammed, Bl\"omker \& Klepel  \cite{MoBlKl:13} obtained in this case
\[
\partial_{T}A=4\partial_X^2 A+\nu A +\tfrac32\sigma^2 A - 3|A|^2 A,
\]
which was already predicted by Hutt et.al.~\cite{Hutt3, Hutt1}
by using a formal centre manifold reduction.
 \end{remark}


\subsection{Nonlinear vs Linear}


In this paper we study the linear case, as the first step towards a full theory of modulation equations.
Also in the already cited results on large domains by Bl\"omker, Hairer \& Pavliotis~\cite{BlHaPa07} this is an essential step towards the full nonlinear result,
which is somewhat separated from the remaining nonlinear estimates,
Although here in weighted spaces the nonlinear estimate does not seem to be that straightforward as the nonlinearity is an unbounded operator.
In~\cite{BlHaPa07}, the authors used a splitting of the solution into a slightly more regular part in $H^1$
and a Gaussian part that was only bounded in $C^0$, but allowed for better estimates due to its Gaussian nature.
Here we focus on the Gaussian part only,
but in contrast to~\cite{BlHaPa07} we face the additional problem that solutions 
and thus error terms are immediately unbounded  in the spatial direction for $|x|\to\infty$.

Let us first state the mild formulation of~\eqref{e:SH}:
\begin{equation}
 \label{e:SHmild}
 u(t)= e^{t\cL_\varepsilon} u_0 - \int_0^t e^{(t-s)\cL_\varepsilon} u(s)^3 \d s + W_{\cL_\varepsilon}(t)\;,
\end{equation}
where the semigroup $e^{t\cL_\varepsilon}$ generated by $\cL_\varepsilon =-(1+\partial_x^2)^2 + \nu\varepsilon^2$
and the stochastic convolution $W_{\cL_\varepsilon}(t)$, which is the solution of the linear equation,
are defined and discussed in more detail in later sections.

For a result on modulation equations we need to compare this to the  mild formulation of~\eqref{e:AE},
given by 
\begin{equation}
 \label{e:AEmild}
A(T)= e^{T(4\partial_X^2+\nu)} A_0 - \int_0^T e^{(T-S)(4\partial_X^2+\nu)}3A(S)|A(S)|^2 \d S + \mathcal{W}_{4\partial_X^2+\nu}(T),
\end{equation}
with the semigroup $e^{T(4\partial_X^2+\nu)}$ and the corresponding 
stochastic convolution $ \mathcal{W}_{4\partial_X^2+\nu}(T)$
defined in terms of a complex-valued Wiener process $\mathcal{W}$.

The main result is to show that
\[
u(t,x) - [ A(\varepsilon^2 t,\varepsilon x)\cdot e^{ix} + c.c. ]\quad\text{is small,} 
\]
and to do so, we have three key steps:
\begin{itemize}
 \item \emph{Nonlinearity} - We need to show that we can control the difference between the nonlinear terms 
 in~\eqref{e:AEmild} and~\eqref{e:SHmild}. This should be similar although quite technical 
 to the deterministic case and we do not treat this here.
  \item \emph{Initial Conditions} - For the two terms containing the initial conditions $u_0$ and $A_0$, 
  we will see that they split in a more regular part that is treated by the known deterministic results 
  and a less regular Gaussian part, which is not good enough to be treated by standard deterministic methods. 
  We discuss these estimates in detail in the proof of Theorem \ref{thm:fullmain}.
   \item \emph{Stochastic Convolution} - The difference between these terms is the main new result of this paper, Theorem~\ref{thm:M}. 
   It is the key estimate to prove a full approximation result for stochastic modulation equations on $\R$.   
\end{itemize}
Let us remark that we focus on bounds in sufficiently good norms  here. 
We might be able to give much simpler bounds in $L^2_\text{loc}$-spaces,
but then we would not be able to control in this norm the cubic $-u^3$ later. 
Thus we focus on the supremum-norm, which is a good compromise.
It does not require any order of 
spatial derivatives, but it still good enough to bound the nonlinearity.
Unfortunately, in our weighted spaces, the nonlinearity is always an unbounded operator,
so some  more care will be needed here.


\subsection{Structure of the paper}


In Section \ref{sec:SC}, we introduce the stochastic convolution,
and discuss its rescaling to the slow time-scale. 
We can already identify all error terms that need to be handled in the following sections.
We do not follow the approach of Walsh~\cite{Wal86} but the one of Da~Prato \& Zabzcyck~\cite{DPZa:14}, using 
explicit series expansions for calculations.

Before getting to the main results, in Section \ref{sec:techlem} we present 
the key technical results for the stochastic convolution and the Gaussian initial conditions.  
We have an  error estimate in spatially weighted $C^0$ spaces
and its extension to estimates in  space and time.
The main assumptions are bounds on the Fourier kernel of the convolution operator, which are provided in the final three sections~\ref{sec:errspace} -- \ref{sec:errgauss}.

Section \ref{sec:main} provides the main results of the paper. First we establish the approximation result for the
for the Ornstein-Uhlenbeck process (stochastic convolution), then we provide the key steps for the full attractivity 
and approximation results for the linear stochastic equation.

Sections~\ref{sec:errspace} -- \ref{sec:errgauss}  provide, as already mentioned, the technical bounds on specific Fourier kernels 
that are necessary to apply the results of Section~\ref{sec:techlem} to the main results.
In Section~\ref{sec:errspace} we present the bounds necessary   
in order to derive estimates in space. 
 In Section~\ref{sec:errtime} we provide the  technical results necessary to obtain space-time estimates for the stochastic convolution.
 Finally, in Section~\ref{sec:errgauss} we show the technical estimates that are necessary to treat the Gaussian part in the initial conditions.

 \section{Stochastic convolution}
 \label{sec:SC}
 
 
 In this section we study the stochastic convolution introduced in~\eqref{e:SHmild}.
 First we introduce all spaces and definitions that we need for our analysis
 and then give the precise definition of the stochastic convolution. Finally we rescale it to the slow time-scale.
 This is necessary to identify the Wiener process driving the modulation equation, 
 and provides the ansatz which error terms need to be bounded.

 \subsection{Notation and Definition}
 
  
 In this part we present basic notation and definitions.
 We introduce all spaces used in the following and define a semigroup
 generated by our differential operator in terms of Fourier-multipliers.
 
 For some small $\gamma>0$ we define the norm
 \[
 \|u\|_{C_\gamma^0}= \sup_{x\in\R} \Big\{ (1+x^2)^{-\gamma/2}|u(x)| \Big\}
 \]
 and denote by $C_\gamma^0$ the space of all locally continuous functions $u:\R\to\R$,
 such that 
 $\|u\|_{C_\gamma^0}<\infty$. Analogously, we denote by $C_{\gamma,T}^0$ the space 
 of locally continuous functions 
 $u:[0,T]\times\R\to\R$ such that the following norm is finite
  \[
 \|u\|_{C_{\gamma,T}^0}= \sup_{s\in [0,T]} \sup_{x\in\R} \Big\{ (1+x^2)^{-\gamma/2}|u(s,x)|\Big\}.
 \]
 Furthermore, we define the space $\cC_{\gamma,T}^0$ by the norm
   \[
 \|u\|_{\cC_{\gamma,T}^0}= \sup_{s\in [0,T]} \sup_{L\in\N} \Big\{L^{-\gamma} \|u\|_{C^0([-L,L])} \Big\}.
 \]
We also use the time independent version  $\cC_{\gamma}^0$.
 \begin{lemma}
  The norms $ \|\cdot\|_{C_{\gamma,T}^0}$ and  $\|\cdot\|_{\cC_{\gamma,T}^0}$ are equivalent for all $\gamma>0$ and all $T>0$
  with constants depending only on $\gamma$.
 \end{lemma}
\begin{proof}
Let us remark that it is sufficient to verify the equivalence of   $ \|\cdot\|_{C_{\gamma}^0}$ and  $\|\cdot\|_{\cC_{\gamma}^0}$.
 First, for $x\in[-L,L]$, 
 \[
(1+x^2)^{-\gamma/2} \geqslant (1+L^2)^{-\gamma/2} \geqslant  2^{-\gamma/2} L^{-\gamma}
 \]
 and thus 
  \[
 \sup_{x\in[-L,L]} (1+x^2)^{-\gamma/2}|u(x)| \geqslant  2^{-\gamma/2}  L^{-\gamma} \|u\|_{C^0([-L,L])}
 \]
 which easily implies the first bound. For the other bound just note that 
$x\in[-1,1]$ implies 
 \[
 (1+x^2)^{-\gamma/2}|u(x)| \leqslant \|u\|_{C^0([-1,1])} \leqslant  \|u\|_{\cC_{\gamma}^0},
 \]
 and $|x|\in[L,L+1]$ implies
 \[
(1+x^2)^{-\gamma/2}|u(x)| \leqslant  L^{-\gamma} \|u\|_{C^0([-L-1,L+1])} \leqslant  2^\gamma  \|u\|_{\cC_{\gamma}^0}.
 \]
\end{proof}
Note that for $\gamma<\rho$ we obviously have the following continuous embeddings
\[
 C_{\rho,T}^0 \subset C_{\gamma,T}^0 
 \quad\text{and}\quad 
 C_{\rho}^0  \subset  C_{\gamma}^0.
\]
Thus a bound on $ C_{\gamma}^0$ for any small $\gamma$ already provides bounds for all larger $\gamma$.

 For small $\gamma>0$ these spaces are (up to a small $\varepsilon$-dependent constant)
 almost invariant under the rescaling $x\to \varepsilon x$.
The following result is straightforward to verify.
\begin{lemma}
 \label{lem:chsp}
For any $u\in{\cC_{\gamma,T}^0}$ one has
\[
c\varepsilon^\gamma \|u \|_{\cC_{\gamma,T}^0} 
\leqslant  \|u(\varepsilon \cdot) \|_{\cC_{\gamma,T}^0} 
\leqslant  C\|u\|_{\cC_{\gamma,T}^0}\;.
\]
 \end{lemma}
The key observation is here that after the substitution $z=x\varepsilon$ one has 
 \[ 
 \varepsilon^\gamma (1+z^2)^{-\gamma/2} 
 \leq  (1+z^2\varepsilon^{-2})^{-\gamma/2} 
 \leq  (1+z^2)^{-\gamma/2} \;.
 \]
%
%
 \subsection{Local \texorpdfstring{$H^1$}{H1}-spaces}\label{sec:localH1}
 %
 %
We show that we can continuously embed $H^1_{\ell,u}$  (see below for the definition) into 
the weighted space
$C_{\rho}^0$ for any positive $\rho$, as it is already embedded into the space $C^0_b(\R)$ of 
continuous bounded functions which embeds into any $C^0_\rho$ with $\rho>0$: 
\begin{eqnarray*}
  \| u \|_{C_{\rho}^0} & = & \sup_{x\in\R} \Big\{ \frac{1}{(1 + x^2)^{\rho / 2}}  | u (x)|\Big\}\\
  \text{[equiv.]} & \leqslant & \sup_{\N \ni L > 0} \Big\{L^{- \gamma}  \| u \|_{C^0 ([- L, L])}\Big\}
  \leqslant  \sup_{L \in \N}  \| u \|_{C^0 ([- L, L])}
  =  \sup_{L \in \Z}  \| u \|_{C^0 ([L, L + 1])}\\
  \text{[Sobolev]}& \leqslant & C\cdot \sup_{L \in \Z}  \| u \|_{H^1 ([L, L + 1])}
 = C\cdot \| u \|_{H^1_{\ell, u}}  \quad \text{[by definition]} .
\end{eqnarray*}
 %
 %
  \subsection{Semigroups and Green's function}
 %
 %
 Here we recall well known facts about semigroups in terms of Green's functions and Fourier multipliers.
 Fix $G_t(x)$ to be the Green's function (fundamental solution) 
 associated to the differential operator 
 \[
 \mathcal{L}=-(1+\partial_x^2)^2.
 \]
 The semigroup $e^{t\cL}$ generated by $\cL$ is thus given as 
 \[
 e^{t\cL}f(x) =  G_t \ast f(x) = \int_{\R} G_t(x-y)f(y) \d y. 
 \]
We can write down $G$ explicitly.
Using the Fourier transform we immediately see that $G_t=\cF^{-1}g_t$ with $g_t(k)= e^{-t(1-k^2)^2}$,
and thus
\begin{align}
\label{e:Green}
G_t(x) 
&= \int_\R g_t(k)e^{ikx}dk  =\frac{1}{2\pi}\int_\R e^{-t(1+k)^2(1-k)^2}e^{ikx}\d k.
\end{align}
 
 This is similar for the operator $\cL_\varepsilon=-(1+\partial_x^2)^2+\nu\varepsilon^2$.
It has the fundamental solution:
\[ 
e^{t (\mathcal{L}+ \varepsilon^2 \nu)} f = G_{\varepsilon, t} \ast f, 
\]
and via Fourier transform:
\[ 
G_{\varepsilon, t} (x) = \frac{1}{2 \pi}  \int_{\R} e^{- t (1 - k^2)^2 + t \varepsilon^2 \nu} e^{ikx} \d  k, 
\]
with kernel  $g_{\varepsilon, t}(k) = e^{- t (1 -
   k^2)^2 + t \varepsilon^2 \nu} $.
 %
 %
\subsection{Properties of the semigroup}
%
%
The semigroup $e^{t\cL}$ generated by $\cL$ is a strongly continuous 
semigroup of linear operators~\cite{Pazy}. Here we rely mainly on the explicit representation
described in the section above. First we recall a bound on $e^{t \mathcal{L}}$ in $L^\infty$-topology
and extend it to the weighted spaces.

Using a result of~\cite[Lemma 2.1]{CoEc:02} we obtain that for all $\beta>0$ 
there is a constant $C>0$ such that  
\[
t^{1/4} |G_t(x)| e^{(\beta+2t^{-1/4})|x|}\leqslant C \quad \text{for all } t\in[0,1]
\]
and thus~\cite[Remark 2.2]{CoEc:02}
\begin{equation}
\label{e:GrBo}
\int_\R |G_t(x)| e^{\beta|x|} \d x \leqslant C\;.
\end{equation}
With these bounds and some less optimal bounds for $t\geqslant1$ from~\cite{MoBlKl:13},
we immediately obtain the following Lemma:
\begin{lemma}
\label{lem:bSG}
For $\gamma\in[0,1)$ there is a constant $C>0$ such that for all $t\geqslant0$ and all  $u\in C^0_\gamma$ 
\[
\|e^{t\cL_\varepsilon}u \|_{C^0_\gamma} \leqslant C \max\{1,t^{\gamma/2}\} \| u \|_{C^0_\gamma}\;.
\]
\end{lemma}
\begin{proof}
First it is easy to see that we need to bound 
\[
\sup_{x\in\R} (1+x^2)^{-\gamma/2} \int_\R |G_t(x-y)| (1+y^2)^{\gamma/2} \d y 
= \sup_{x\in\R}  \int_\R \Big(  \frac{1+(x-z)^2}{1+x^2} \Big)^{\gamma/2}  |G_t(z)|  \d z\;.
\]
For $t\in[0,1]$ the result now follows from~\eqref{e:GrBo} and
\[
  \frac{1+(x-z)^2}{1+x^2} \leqslant 2 (1+z^2)\;.
\]
For $t>1$ we use that from~\cite{MoBlKl:13} 
\[
G_t(z)=t^{-1/2} g_{t^{-1/2}}(t^{-1/2}z)
\quad \text{with} \quad
\sup_{\xi\in\R} \sup_{\tau\in(0,1)} |g_\tau(\xi)| (4+\xi^2) \leqslant C\;.
\]
Thus 
\[
\begin{split}
 \int_{\R}(1+z^2)^{\gamma/2} |G_t(z)| \d z &\leqslant C  \int_{\R}(1+z^2)^{\gamma/2} t^{-1/2} (4+z^2/t)^{-1} \d z\\
 &\leqslant C \int_{\R} \frac{(1+t\xi^2)^{\gamma/2}}{4+\xi^2}   \d \xi \leqslant C t^{\gamma/2}.
\end{split}
\]
 \end{proof}
Let us remark without proof that the restriction to $\gamma \in(0,1)$ in the last step of the proof above
does not seem to be necessary, 
as we could use any power of $(4+z^2)$. 
%
%
%
 \subsection{Definition of cylindrical Wiener process}
 %
 %
 Let's fix for the whole paper an abstract probability space $(\Omega,\cA,\P)$ 
 on which all stochastic processes are defined.
 Following~\cite{DPZa:14}  we define:
 
 \begin{defi}
 \label{def:WP}
 A \emph{standard cylindrical Wiener process} $W(t)$ 
 is given by any orthonormal basis $\{e_\ell\}_{\ell\in\N}$ of $L^2(\R,\R)$
 and any family of (real valued) i.i.d. standard Brownian motions $\{B_\ell\}_{\ell\in\N}$
 such that
 \begin{equation}
 \label{e:WP}
   W(t)= \sum_{\ell\in\N} B_\ell(t) e_\ell\;.
 \end{equation}
 \end{defi}
Obviously, a cylindrical Wiener process is not an $L^2(\R)$-valued random variable.
It is just defined in a larger space. For details see~\cite{DPZa:14}.
Moreover, it is characterized by being a Gaussian process such that 
for all $u,v\in L^2(\R,\R)$ and all $t,s\geqslant 0$
\[
\E\langle W(t),u \rangle =0 
\quad\text{and}\quad 
\E \Big\{\langle W(t),u \rangle \langle W(s),v \rangle \Big\} 
=\min\{t,s\}\langle u,v \rangle \;.
\]
A sometimes confusing fact is that  for every fixed $t$ the process
$\{W(t,x)\}_{x\in\R}$ is a stationary and thus translation invariant process,
but the  basis functions in which we expand are not at all 
translation invariant and might decay fast at infinity like Hermite functions.

We also need the notion of a standard $\C$-valued cylindrical process,
which is not as standard:
 \begin{defi}
 A \emph{complex-valued standard cylindrical Wiener process} $\mathcal{W}(t)$ 
 is given by any orthonormal basis $\{e_\ell\}_{\ell\in\N}$ of $L^2(\R,\C)$
 and any family of $\C$-valued i.i.d. standard Brownian motions $\{\beta_\ell\}_{\ell\in\N}$
 such that
 \begin{equation}
 \label{e:WPc}
  \mathcal{W}(t)= \sum_{\ell\in\N} \beta_\ell(t) e_\ell\;.
 \end{equation}
 \end{defi}
Note that for a $\C$-valued standard Brownian motion $\beta_\ell$ one has
\[
\E \beta_\ell(t)^2 =0 \quad \text{and} \quad \E |\beta_\ell(t)|^2 = t\;. 
\]
\begin{es}
 If we take $2$ independent copies $W^{(j)}$, $j=1,2$ 
 of real valued standard cylindrical Wiener processes
 in the sense of Definition \ref{def:WP},
 then 
 \[
 \mathcal{W}(t):= \frac1{\sqrt{2}} \Big[W^{(1)}(t) + i W^{(2)}(t)\Big]
 \]
 is a complex valued standard cylindrical Wiener process with same 
 orthonormal basis $\{e_\ell\}_{\ell\in\N}$ of both  $L^2(\R,\C)$ and $L^2(\R,\R)$
 and  with Brownian motion 
 \[
 \beta_\ell = \frac{[B^{(1)} + i B^{(2)}]}{\sqrt{2}}.
 \]
\end{es}

We will model space-time white noise always as the derivative of 
a standard cylindrical Wiener process. 
There is also the equivalent approach of Walsh \cite{Wal86}
 by looking at the derivative of a Brownian sheet. 
But Dalang \& Quer-Sardanyons \cite{DaQS11} and also \cite{FI15} showed that both 
formulations yield, up to taking the right versions of the processes, the same integrals.
\begin{remark}[Coloured Noise]
As already mentioned previously, we only consider space-time white noise here in this paper. 
If we would consider  a $Q$-Wiener process $W$, for example
with $Q$ given by the convolution against some function $q$,
then we obtain for the corresponding noise $\xi=\partial_t W$ 
that it is a generalized centred Gaussian process with correlation
\[
\E \xi(t,x)\xi(s,y)=\delta(t-s)q(x-y).
\]
So we are still in the case of homogeneous noise which is translation invariant. 
Referring to~\cite{Blo05} the stochastic convolution is more regular in space, 
but we expect it to be still unbounded for every fixed $t>0$.

In~\cite{BlHaPa07} in the case of large, but bounded, domains the case of spatially smoother coloured noise is also treated,
but the amplitude equation still displays space-time white noise. 
In order to avoid further technicalities in rescaling coloured noise, 
we do not address this issue in this paper.
\end{remark}
%
 %
 \subsection{Stochastic Integrals}
 %
 %
%
For a deterministic Hilbert-Schmidt-operator valued function 
$\cH \in L^2([S,T],\HS(L^2(\R)))$, one can define the stochastic integral
with respect to the real or complex cylindrical Wiener process $W$, which we can also expand in the basis $e_\ell$
that were  used to define the Wiener process in~\eqref{e:WP}) or~\eqref{e:WPc}: 
\[
\int_S^T \cH(s)dW(s)= \sum_{\ell\in\N} \int_S^T \cH(s) e_\ell \d\beta_\ell(s).
\]
By It\^o-isometry we have 
\[
\E\Bigg\|\int_S^T \cH(s)\d W(s)\Bigg\|^2_{L^2(\R)}= \sum_{\ell\in\N}  \int_S^T\| \cH(s) e_\ell \|^2_{L^2(\R)} \d s = \|\cH\|^2_{L^2([S,T],\HS(L^2(\R))}
\]
 %
 \subsection{Stochastic convolution}
 For a cylindrical Wiener process $W$ defined in~\eqref{e:WP}
 and the semigroup generated by $\cL$
 we can define the stochastic convolution
 \begin{equation}
 \label{e:sc}
 \begin{split}
 W_\mathcal{L}(t,\cdot)
 &\coloneqq\int_0^t  e^{(t-s)\cL}\d W(s)
 = \sum_{\ell\in \N} \int_0^t e^{(t-s)\cL}e_\ell( \cdot)\d\beta_\ell(s)\\
 &= \sum_{\ell \in \N} \int_0^t \int_\R G_{t-s}(\cdot -y)e_\ell(y)\d y\,\d\beta_\ell(s).
 \end{split}
  \end{equation}
 This is by definition  (see~\cite{DPZa:14}) the mild solution to the linear problem
 \[
 \d u =\cL u \d t +\d W, \qquad u(0)=0.
 \]
 In this section we focus on $W_\cL(t,x)$, but the results are true for all other 
 stochastic convolutions considered in this paper, 
 like for example $\mathcal{W}_{4\partial_X^2}(T,X)$ on the slow time-scale, with a complex Wiener process $\mathcal{W}$.
 
 The following Lemma is well known  in the setting of Walsh and easy to verify here. 
 We give a brief sketch of a proof for completeness of presentation.
\begin{lemma}
\label{lem:well}
 The stochastic convolution $ W_\mathcal{L}(t,x)$
 in~\eqref{e:sc} is for all $t\geqslant0$ and $x\in\R$ a well 
defined real-valued Gaussian  random variable with mean $0$ and variance $\int_0^t \| g_{s}\|^2 ds$.  
\end{lemma}
\begin{proof}
 It is easy to check that the series in~\eqref{e:sc} is 
 for all $(t,x)\in[0,T]\times\R$ a Cauchy-sequence in $L^2(\Omega,\R)$ with
\begin{align*}
 \E | W_\mathcal{L}(t,x) |^2 
 &=  \sum_{\ell \in \N}  \int_0^t \Big|\int_\R G_{t-s}(x -y)e_\ell(y)\d y \Big|^2 \d s \\
 \textrm{[Parceval]}\quad& =  \int_0^t \| G_{t-s}(x - \cdot) \|^2 \d s = \int_0^t \| G_{s}\|^2 \d s \\
\textrm{[Plancherel]}\quad & =  \int_0^t \| g_{s}\|^2 \d s <\infty,
\end{align*}
where a straightforward calculation shows that the last term is finite.
Gaussianity follows from the fact that Gaussians are closed under mean-square convergence.
\end{proof}

\begin{remark}
Lemma \ref{lem:well} immediately implies that for all $T>0$ and $L>0$  
\[
 W_\mathcal{L} \in L^p(\Omega\times[0,T]\times[-L,L]) \;.
\]
With the methods of section \ref{sec:techlem} we will see that 
$W_\mathcal{L} \in L^p(\Omega, C^0([0,T]\times[-L,L]))$ 
with norm growing in $L$ 
slower than  any small power of $L$.
\end{remark}

The term $W_\cL$ behaves like the square root of logarithm at $x\to\infty$, 
a property we only state as a remark without proof. 
 \begin{remark}\label{rem:conjlog} For all $t>0$
   \[
    \lim_{L\to \infty} \frac1{\sqrt{\log(L)}}\sup_{x\in[-L,L]}|W_\mathcal{L}(t,x)| \in (0,\infty) \;. 
   \]
 \end{remark}

This result has been obtained for Gaussian processes by Qualls and Watanabe~\cite{quawat71}.
The key idea required in proving this remark is just the fact that for fixed $t>0$ we
are considering a Gaussian field with constant variance (see Lemma~\ref{lem:well}).
To conclude the proof one needs to compute and analyse the covariance function, which is a
 straightforward but quite technical task.

\subsection{Rescaling}
 

 In this section we rescale the stochastic convolution $W_\cL$
 to the slow time- and space-scale
 \[
 X=\varepsilon x,\quad  T=\varepsilon^2 t.
 \] 
 We want to focus on the
 slow time scale and large spatial scale, forgetting about the local phenomena
induced by fast oscillations and focusing only on the slow modulation of these fast pattern.  

Before we start, we  rescale the Wiener process first, and show that 
such rescaled Wiener process is at least in law still the same cylindrical Wiener process.
Note that $\stackrel{\mathrm{(d)}}{=}$ denotes equality in law. 
 \begin{lemma} For a real or complex valued standard cylindrical Wiener process $W$ we have 
 \[
 W(t,x)\stackrel{\mathrm{(d)}}{=} W^{(\varepsilon)}(T,X)\stackrel{\mathrm{def}}{=} \sum_{\ell \in\N} \beta_\ell^{(\varepsilon)}(T)e_\ell^{(\varepsilon)}(X),
 \]
 where  $\{e_\ell\}_{\ell\in\N}$ is an  orthonormal basis in $L^2(\R)$,
 $\{\beta_\ell\}_{\ell\in\N}$ a family of (real or complex) i.i.d. standard Brownian motions 
 and we defined for some $c\in\R$:
 \begin{equation*}
 e_\ell^{(\varepsilon)}(X)=\varepsilon^{-1/2} e^{icx} e_\ell(X\varepsilon^{-1})\;, \qquad
 \beta_\ell^{(\varepsilon)}(T)= \varepsilon\beta_\ell(T\varepsilon^{-2})\;.
 \end{equation*}
 \end{lemma}
\begin{proof}
Check that $e_\ell^{(\varepsilon)}$ is an orthonormal basis of $L^2(\R)$. 
By the scaling properties of Brownian motion since $\beta_\ell(t)$ is a family of i.i.d. Brownian motions,  $\alpha^{-1}\beta_\ell(\alpha^{2}t)$ is one, too.
Thus by definition $W^{(\varepsilon)}$ is a standard cylindrical Wiener process.
\end{proof} 
 The full rescaling of the Wiener process is stated in the following Lemma.
 Here we do not start by  $W_\cL$, but consider the final rescaled  result, 
 and assume that is independent of $\varepsilon$, 
 then we construct a proper $\varepsilon$-dependent Wiener process for  $W_\cL$,
 such that the rescaling is true.
 \begin{lemma}[Rescaling Lemma]
 \label{lem:res}
  Fix    an  orthonormal basis  $\{e_\ell\}_{\ell\in\N}$ of $L^2(\R,\C)$
 and $\{\beta_\ell\}_{\ell\in\N}$ a family of $\C$-valued i.i.d. standard Brownian motions.
 Then there exists an $\varepsilon$-dependent standard real-valued cylindrical 
 Wiener process $W^{(\varepsilon)}$ such that
 \begin{align*}
 \lefteqn{\varepsilon^{1/2}W^{(\varepsilon)}_\mathcal{L}(T\varepsilon^{-2},X\varepsilon^{-1})=}\\
 & =\frac{1}{2\pi} \sum_{\ell \in \N} \int_0^T\int_\R\int_{-1/\varepsilon}^\infty e^{-(T-S)(2+k\varepsilon)^2k^2}e^{ik(X-Y)}\d k
  e_\ell(Y)\ \d Y\d \beta_\ell(S) e^{iX\varepsilon^{-1}}+ \mathrm{c.c.}
 \nonumber
 \end{align*}
 \end{lemma}

 \begin{proof}
We define the complex Wiener process 
\[
\mathcal{W}(T,X) = \sum_{\ell \in\mathbb{N}} e_\ell(X)\beta_\ell(T).
\]
Define the projection onto the positive wave-numbers as the  operator $P^+$ given by
\[
P^+f(x) = \frac1{2\pi} \int_0^\infty \hat{f}(k) e^{ikx} \d k.
\]
We have the following identities, first for convolutions
\[
 P^+ f\ast g =\mathcal{F}^{-1} [\hat{ P^+ f}  \hat{g}]  = \mathcal{F}^{-1} [\hat{f} \chi_{[0,\infty)}  \hat{g}] =    f\ast P^+ g , 
\]
and then, using Plancherel, for scalar-products
\[\langle P^+u,v\rangle = \langle u,P^+v\rangle  = \langle P^+u,P^+v\rangle \;.
\]
Let us now start with the rescaling 
\[
\frac1{2\pi}\sum_\ell  \int_0^T \int_\R \int_{-1/\varepsilon}^\infty  
e^{-(T-S)k^2(2+\varepsilon k)^2} e^{ik(X-Y)} \d k e_\ell (Y) \d Y \d\beta_\ell(S) e^{iX/\varepsilon } + c.c. 
\]
substituting $k'=(1+k\varepsilon)$ we get 
\[ 
= \frac1{2\pi}\sum_\ell  \int_0^T \int_\R  \frac1{\varepsilon}\int_{0}^\infty  
e^{-(T-S)\varepsilon^{-2} (1- k^2)^2} e^{ik(X-Y)/\varepsilon} \d k e^{iY/\varepsilon }e_\ell (Y) \d Y \d\beta_\ell(S)  + c.c. 
\]
plugging in $y=Y/\varepsilon$ and $x=X/\varepsilon$ yields 
 \[ 
= \frac1{2\pi}\sum_\ell  \int_0^T \int_\R  \int_{0}^\infty  
e^{-(T-S)\varepsilon^{-2} (1- k^2)^2} e^{ik(x-y)} \d k e^{iy }e_\ell (y\varepsilon) \d y \d\beta_\ell(S)  + c.c. 
\]
 substituting now $t=T/\varepsilon^2$ and $s=S/\varepsilon^2$ gives
  \[ 
= \frac1{2\pi}\sum_\ell  \int_0^t \int_\R  \int_{0}^\infty  
e^{-(t-s)(1- k^2)^2} e^{ik(x-y)} \d k e^{iy }e_\ell (y\varepsilon) \d y \d\beta_\ell(s \varepsilon^2)  + c.c. 
\]
Using a rescaled complex Wiener process we have
  \[ 
= \varepsilon \frac1{2\pi}\sum_\ell  \int_0^t \int_\R  \int_{0}^\infty  
e^{-(t-s)(1- k^2)^2} e^{ik(x-y)} \d k e^{iy }e_\ell (y\varepsilon) \d y \d\tilde\beta_\ell(s )  + c.c. 
\]
and, denoting the Green's function by $G_t(x)$, this gives 
  \[ 
= \varepsilon\sum_\ell  \int_0^t \int_\R  P^+ G_{t-s}(x-y) e^{iy }e_\ell (y\varepsilon) \d y \d\tilde\beta_\ell(s )  + c.c. 
\]
We define now $f_\ell(y)=  \varepsilon^{1/2} e^{i y }e_\ell (\varepsilon y )$, 
and $\{f_\ell\}_{\ell\in\mathbb{N}}$ is again an orthonormal basis
so that we obtain
 \[ 
= \varepsilon^{1/2} \sum_\ell  \int_0^t   P^+ G_{t-s} \ast  f_\ell  (x) \d\tilde\beta_\ell(s )  + c.c. 
\]
and by moving the $P^+$ it becomes
 \[ 
= \varepsilon^{1/2} \sum_\ell  \int_0^t   G_{t-s} \ast P^+f_\ell  (x) \d\tilde\beta_\ell(s )  + c.c. 
= \varepsilon^{1/2} \sum_\ell  \int_0^t   e^{(t-s)\mathcal{L}} P^+f_\ell  (x) \d\tilde\beta_\ell(s )  + c.c. 
\]
This is equal to  $\varepsilon^{1/2} W^{(\varepsilon)}_{\mathcal{L}}(t,x)$ if we can verify that 
\[
W^{(\varepsilon)}(t,x) = \sum_\ell P^+f_\ell  (x) \tilde\beta_\ell(t ) +c.c.
\]
 is a standard cylindrical Wiener process. Note that this would be trivial, in case  
 \[
 \sum_\ell f_\ell  (x) \tilde\beta_\ell(t)
 \]
 was a real valued Wiener process, as $P^+ u +c.c. =u$ for any real-valued function. 
 
First we see that $W^{(\varepsilon)}$ is a centred Gaussian process, 
we only need to check that the covariance-operator is the identity.
\begin{eqnarray*}
\lefteqn{\mathbb{E} \langle W^{(\varepsilon)}(t) ,u \rangle \langle W^{(\varepsilon)}(t),  v\rangle }\\
&=&\sum_\ell   \mathbb{E}  \langle  P^+f_\ell \tilde\beta_\ell(t ) +c.c. ,u \rangle \langle  P^+f_\ell \tilde\beta_\ell(t) +c.c.  , v\rangle ]\\
&=&t \sum_\ell  \langle  \overline{P^+f_\ell} ,u \rangle \langle  P^+f_\ell  , v\rangle 
+t \sum_\ell  \langle  P^+f_\ell ,u \rangle \langle  \overline{P^+f_\ell} , v\rangle \\
&=&t \sum_\ell  \overline{ \langle P^+f_\ell ,u \rangle } \langle  P^+f_\ell  , v\rangle 
+t \sum_\ell  \langle  P^+f_\ell ,u \rangle  \overline{\langle P^+f_\ell , v\rangle }\\
&=&t \sum_\ell  \overline{ \langle f_\ell ,P^+ u \rangle } \langle  f_\ell  ,P^+ v\rangle 
+t \sum_\ell  \langle  f_\ell ,P^+u \rangle  \overline{\langle f_\ell , P^+v\rangle }
\end{eqnarray*}
where we used that $u$ and $v$ are real in order to pull out the complex conjugate from the scalar-product.
Now we use Parceval 
\[
\langle f,g\rangle = \sum_\ell \langle f,f_\ell\rangle\overline{\langle g,f_\ell\rangle}
\]
to obtain 
\begin{eqnarray*}
\mathbb{E} \langle W^{(\varepsilon)}(t) ,u \rangle \langle W^{(\varepsilon)}(t),  v\rangle 
&=&t \sum_\ell \langle P^+ u ,  f_\ell\rangle \overline{  \langle  P^+ v ,f_\ell \rangle } 
+t \sum_\ell \overline{ \langle  P^+u , f_\ell\rangle }  \langle P^+v, f_\ell \rangle \\
&=& t  \langle P^+ u ,   P^+ v \rangle  
+t  \langle   P^+v,   P^+u \rangle \\
&=& t  \langle  P^+ u , v \rangle  +t  \overline{\langle P^+u   ,  v \rangle} \\
&=& t  \langle  P^+ u , v \rangle  +t  \langle \overline{P^+u }  ,  v \rangle\quad\text{ as $v$ is real}  \\
&=& t  \langle   u  , v \rangle 
\end{eqnarray*}
as $ P^+ u + \overline{P^+u } =u $ for a real-valued function $u$.
 \end{proof}


 Now we observe that on the RHS in Lemma \ref{lem:res}
 we almost have the stochastic convolution of an operator,
 if we suppose that $\varepsilon k$ is small and just had the innermost integral over $\R$, 
 instead that just $(-1/\varepsilon,+\infty)$. 
 In that case we could have written 
 \begin{equation}
 \label{e:approx}
   \varepsilon^{3/2}W^{(\varepsilon)}_\mathcal{L}(t,x) \approx \varepsilon W_{4\partial_x^2}(T,X)e^{ix} + \mathrm{c.c.}
 \end{equation}
 What we want to do in the next sections is to give a bound on the error, the difference between what we have and what we want to use.
 Therefore, we provide first in Section~\ref{sec:techlem} the technical Lemmas  
 that reduce the task to calculations on Fourier-kernels of convolution operators.

 
 \section{H\"older Estimates for general convolution integrals}
 \label{sec:techlem}
 
 
In~\eqref{e:approx} all error terms are of the type 
$ \int_0^T \cH_{T-\tau} dW(\tau)$, where   $\cH_\tau$ is a convolution operator 
written in terms of the Fourier-transformed kernel.
For the initial conditions we also need bounds on $H_\tau A$, with $A$ being a Gaussian function. In this section we provide bounds on these objects in terms of norms of the kernel.

First we need the following key Lemma for estimates:

\begin{lemma}\label{lem:sumuplem}
Let $W$ be a complex valued standard cylindrical Wiener process 
with orthonormal basis $\{e_\ell\}_{\ell\in\N}$ in $L^2(\R,\C)$
 and $\{\beta_\ell\}_{\ell\in\N}$ a family of $\C$-valued i.i.d.~standard Brownian motions.

  Given a function $f$, its Green's function $H_\tau =
  \mathcal{F}^{-1} f(\tau,\cdot)$, and its corresponding convolution operator $\cH_\tau=H_\tau\ast$
  let us define 
  \[
  \Phi(T)  =   \int_0^T \cH_{T-\tau} dW
\]
expanded as 
\[
  \Phi(T,X)  = \sum_{\ell \in \N} \int_0^T \int_\R H_{T-\tau}
  (X-Y){e}_\ell(Y)\d Y \d {\beta}_\ell(\tau)\;.
\]
Then for all  $p>1$ and all $\gamma>\frac1p$ there is a constant $C>0$ 
such that for all $L\geqslant1$,
\[
\E\|\Phi(T,\cdot) \|^p_{\mathcal{C}^0(-L,L)} \leqslant C\cdot
L^{\gamma p} \|f\|_{H^\gamma}^p,
\]
where the $L^2(0,T,H^\gamma)$-norm of $f$ is defined as  
\[\|f\|^2_{L^2(H^\gamma)} =  \int_0^T \int_\R|f(S,k)|^2(|k|^{2\gamma}+1)\d k \d S.
\]
\end{lemma}

Let us remark that this is actually a slight abuse of notation,
as we look at the $H^\gamma$-norm of the kernel $H_\tau$, which has Fourier transform $f$.

\begin{remark}
Let us remark that similar estimates than the ones presented here 
were derived in \cite{DaSS09,Dal09} for the Green's function of the 
stochastic wave equation and more regular noise. 
Moreover, we do not only need finiteness of the norms, but explicit bounds.
Especially, the dependence of the constants on $L$.
\end{remark}

\begin{remark}
We will see that the conditions above on $\gamma$ are no problem, since
thanks to Gaussianity we can do the estimates for $p=2$ and then send
$p$ to infinity and hence we can choose $\gamma>0$ as small as we want.
\end{remark}

\begin{proof}[Proof of Lemma~\ref{lem:sumuplem}]
  We proceed by steps. We start by using the fractional Sobolev embedding \cite{RuSi96}
  and the explicit representation of the norm in $W^{\alpha,p}$
 \begin{equation*}
  \begin{split}
&\|\Phi(T,\cdot)\|^p_{\mathcal{C}^{0}([-L,L])}  
 = \|\Phi(T, L \cdot)\|^p_{\mathcal{C}^{0}([-1,1])}  
 \leqslant C\|\Phi(T,L\cdot)\|^p_{\mathcal{W}^{\alpha,p}([-1,1])}
\\
&= C\bigg[\|\Phi(T,L\cdot)\|^p_{L^p[-1,1]}+ \iint_{[-1,1]^2}\frac{|\Phi(T,LX)-\Phi(T,LY)|^p}{|X-Y|^{1+\alpha p}}\d X\d Y\bigg] \\
&= CL^{-1}\bigg[ \|\Phi(T,\cdot)\|^p_{L^p[-L,L]}+ L^{\alpha p} \iint_{[-L,L]^2}\frac{|\Phi(T,X)-\Phi(T,Y)|^p}{|X-Y|^{1+\alpha p}}\d X\d Y\bigg]
  \end{split}
 \end{equation*}
with $\alpha$ fixed later such that $\gamma>\alpha >1/p$. Let us remark that by H\"older inequality it is
sufficient to verify the claim only for sufficiently large $p>0$.

For the next step we take the expectation and use the Gaussianity of $\Phi$:
 \begin{equation}
\label{eq:explrep} 
\begin{split}
    \E\|\Phi(T,\cdot)\|^p_{\mathcal{C}^{0}([-L,L])} 
    \leqslant & C L^{-1}\E\|\Phi\|^p_{L^p[-L,L]} \\ 
    &+ CL^{-1+\alpha p} \iint_{[-L,L]^2}\frac{C_p(\E|
    \Phi(T,X)-\Phi(T,Y)|^2)^{\sfrac{p}{2}}}{|X-Y|^{1+\alpha p}}\d X\d Y].
  \end{split}
  \end{equation} 
 By means of this well-known trick we translated our problem from the generic $p$-th moment to the second moment only.
 
 We now proceed, bounding the second moments (first use It\^o-isometry and substitute in time):
 \begin{equation}
  \begin{split}
   \E |\Phi(T,X)-\Phi(T,Y)|^2 &= \sum_{\ell \in \N}     
   \int_0^T \Big| \int_\R[H_{\tau}(X-z)-H_{\tau}(Y-z)]e_\ell(z)\d z \Big|^2\d 
   \tau\\
   \mathrm{[Parceval]}\quad& =
    \int_0^T \|H_\tau(X-\cdot) - H_\tau(Y-\cdot)\|^2_{L^2} \d\tau\\
  \mathrm{[Plancherel]}\quad & = \int_0^T \int_\R |f(\tau,k)(e^{ikX}-e^{ikY})|^2\d k\d\tau\\
   &\leqslant \int_0^T \int_\R |f(\tau,k)|^2|k|^{2\gamma}\d k\d\tau 
   |X-Y|^{2\gamma}\label{eq:termingamma}.
  \end{split}
\end{equation}

We used for the application of Plancherel,  that for $H_T$ being the kernel of $f(T,\cdot)$ one has
\begin{equation}
\label{e:Hdiff}
  \begin{split}
  H_{\tau}(X-z)-H_{\tau}(Y-z)&= \int_\R f(\tau,k)(e^{ik(X-z)}-e^{ik(Y-z)})\d k\\
  &= \int_\R f(\tau,k)(e^{ikX}-e^{ikY})e^{-ikz}\d k\\
  &= \mathcal{F}[f(\tau,k)(e^{ikX}-e^{ikY})]\;.
 \end{split}
\end{equation}
 Now we take~\eqref{eq:termingamma} and we plug it back into~\eqref{eq:explrep}:
 \begin{multline}\label{eq:sumupspace}
\E\|\Phi(T,\cdot)\|^p_{\mathcal{C}^{0}([-L,L])} 
\leqslant C L^{-1} \E\|\Phi\|^p_{L^p[-L,L]} 
\\ 
+ C L^{-1+\alpha p} \|f\|_{L^2(H^\gamma)}^{p} \iint_{[-L,L]^2}\frac{|X-Y|^{\gamma p}}{|X-Y|^{1+\alpha p}}\d X\d Y \;.
\end{multline}
We can compute
\begin{equation*}
  \iint_{[-L,L]^2}\frac{|X-Y|^{\gamma p}}{|X-Y|^{1+\alpha p}}\d X\d Y
  =C \cdot L^{p(\gamma -\alpha)+1}, 
\end{equation*}

so that~\eqref{eq:sumupspace} becomes
\begin{equation}\label{eq:sumupspace3}
  \E\|\Phi(T,\cdot)\|^p_{\mathcal{C}^{0}} 
  \leqslant C L^{-1} \E\|\Phi\|^p_{L^p}
  + C L^{p\gamma} \|f\|_{L^2(H^\gamma)}^{p}   \;.
\end{equation}

Now there is only one thing left: bound $\E\|\Phi\|^p_{L^p}$. To do that, we take the same road followed above for the other term:
  \begin{equation}\label{eq:lpboundf}
   \begin{split}
    \E\|\Phi(T,\cdot)\|^p_{L^p[-L,L]} &= 
    \E \int_{-L}^L|\Phi(T,X)|^p \d X
    \leqslant C_p \int_{-L}^L (\E |\Phi(T,X)|^2)^{\sfrac{p}{2}} \d X\\
    & \leqslant C \int_{-L}^L \Big( \int_0^T \int_\R |f(\tau,k)|^2\d k 
      \d\tau \Big)^{\sfrac{p}{2}} \d X
\leqslant C L \|f\|_{L^2(H^\gamma)}^{p},
   \end{split}
  \end{equation}
 and we can substitute that in~\eqref{eq:sumupspace3}, obtaining the thesis for $p$ large. 
  To finish the proof for any $p>0$, we use H\"older inequality for $q$ sufficiently large:
  \[
    \E\|\Phi(T,\cdot) \|^p_{\mathcal{C}^0(-L,L)} 
    \leqslant (\E\|\Phi(T,\cdot) \|^{pq}_{\mathcal{C}^0(-L,L)})^{\sfrac{1}{q}}
     \leqslant  C L^{\gamma p} \|f\|^p_{L^2(H^\gamma)}\;.\qedhere
  \]
\end{proof}
Let us now extend to an estimate in space and time, where we first focus on a bounded domain in space and time.
\begin{lemma}
 \label{lem:key2}
 Let $\Phi$ be as in Lemma \ref{lem:sumuplem} and define for some $\gamma>0$ and $T>0$
 \begin{multline*}
 \|f\|_{\mathcal{K}_\gamma}^2
 = \sup_{S\in[0,T]} S^{-2\gamma} \int_0^S \|f(\tau,\cdot)\|^2_{L^2} \d\tau +
 \\+   \sup_{0\leqslant R \leqslant S \leqslant T} (S-R)^{-2\gamma}\int_0^R \|f(S-R+\tau,\cdot)-f(\tau,\cdot)\|^2_{L^2} \d\tau \;.
\end{multline*}
Then for all  $p>1$ such that $\gamma>\frac1p$ there is a constant $C>0$ 
such that for all $L\geqslant1$,
\[
\E\|\Phi \|^p_{\mathcal{C}^0([0,T]\times [-L,L])} \leqslant C\cdot
L^{p \gamma} [\|f\|_{H^\gamma}^2 +   \|f\|_{\mathcal{K}_\gamma}^2]^{p/2}.
\]
\end{lemma}
\begin{remark}
Obviously, $\|\cdot\|_{\mathcal{K}_\gamma}$ defines a norm, 
but we do not know whether it is equivalent to a known space. 
The first term has some similarities to Morrey-spaces,
while the second one is an averaged Hölder coefficient.
But we do not need properties of that space, 
we just need to bound explicitly these norms.
\end{remark}
We recall the following Ito-isometry, which we already stated for the stochastic convolution. 
\begin{lemma}
\label{lem:ItoIso}
Let $g$ be a square integrable function in space and time and $W$ a complex standard cylindrical Wiener process. Then 
\[
\E \Big|\int_a^b \mathcal{F}^{-1}g(t)*\d W(t)\Big|^2 = \int_a^b \|g(t,\cdot)\|^2_{L^2}\d t\;,
\]
where the stochastic integral is a function in $x$.
\end{lemma}
\begin{proof}
By direct calculation,
\begin{equation*}
\begin{split}
\E \Big|\int_a^b \mathcal{F}^{-1}g(r)*\d W(r)\Big|^2 = \E \Big|\sum_\ell\int_a^b (\mathcal{F}^{-1}g(r)*e_\ell)\d\beta_\ell\Big|^2\\
= \sum_\ell \int_a^b  |\mathcal{F}^{-1}g(r)*e_\ell|^2 \d r = \sum_\ell \int_a^b  \langle\mathcal{F}^{-1}g(r,x-\cdot),e_\ell\rangle^2 \d r\\
= \int_a^b\|\mathcal{F}^{-1}g(r,x-\cdot)\|_{L^2}^2\d r = \int_a^b \|g(r,\cdot)\|^2_{L^2}\d r \;,
\end{split}
\end{equation*}
where we used the Parceval theorem, translation invariance and Fourier isometry.
\end{proof}
\begin{proof}[Proof of Lemma \ref{lem:key2}.]
We use again fractional Sobolev spaces with $\gamma>\alpha>2/p$ and  Gaussianity to obtain
(with $L^p$-norm in space and time)
\begin{align*}
\lefteqn{ \E \|\Phi \|^p_{\mathcal{C}^0 ([0, T] \times [- L, L])}}\\ 
&\leqslant C  \E \Big[ L^{-1} \|\Phi \|^p_{L^p} 
+ L^{-1+\alpha p} \int_0^T \int_0^T \int_{- L}^L \int_{- L}^L \frac{| \Phi
   (S, X) - \Phi (R, Y) |^p}{[(S - R)^2 + (X - Y)^2]^{1 + \alpha p / 2}}
   \d  X \d  Y \d  S \d  R \Big] \\
&\leqslant C \Big[L^{-1} \E \| \Phi \|^p_{L^p} +L^{-1+\alpha p} 
   \int_0^T \int_0^T \int_{- L}^L \int_{- L}^L \frac{C_p  [\E | \Phi
   (S, X) - \Phi (R, Y) |^2]^{p / 2}}{[(S - R)^2 + (X - Y)^2]^{1 + \alpha p /
   2}} \d  X \d  Y \d  S \d  R \Big] . 
  \end{align*}  
Integrating equation~\eqref{eq:lpboundf} in time treats the $L^p$-norm in space and time, so we only need to look at the fourfold integral.
Here we focus on the second moments in the integrand. Using~\eqref{eq:termingamma}, we obtain
\begin{eqnarray*}
\lefteqn{  \E | \Phi (S, X) - \Phi (R, Y) |^2 }\\
& \leqslant & 2 \E | \Phi
  (S, X) - \Phi (S, Y) |^2 +2\E | \Phi (S, Y) - \Phi (R, Y) |^2\\
  & \leqslant & 2 
 \|f\|^2_{L^2(H^\gamma)}| X - Y |^{2
  \gamma} + 2\E | \Phi (S, Y) - \Phi (R, Y) |^2 .
\end{eqnarray*}
Since the first term has already been treated before,  we focus on the second one. If we can 
prove a bound by  $C  \|f\|_{\mathcal{K}_\gamma}^2 | S - R |^{2 \gamma}$, then we can easily finish the proof 
as in the previous Lemma \ref{lem:sumuplem} with sufficiently large $p$.

Let's
assume without loss of generality $S > R$.
\begin{eqnarray*}
  \E | \Phi (S, \cdot) - \Phi (R, \cdot) |^2 
  & = & \E \Big|
  \int_0^S H_{S - t} \ast \d  W(t) - \int_0^R H_{R - t} \ast \d  W(t)
  \Big|^2\\
  & \leqslant & 2\E \Big| \int_R^S H_{S - t} \ast \d  W (t)
  \Big|^2 + 2\E \Big| \int_0^R (H_{S - t} - H_{R - t}) \ast \d 
  W (t) \Big|^2\;.
\end{eqnarray*}
So we have by Lemma \ref{lem:ItoIso}
\begin{eqnarray*}
  \E | \Phi (S, Y) - \Phi (R, Y) |^2 
  & \leqslant & 2 \int_R^S \!\!\| f({S-t},\cdot) \|^2_{L^2} \d  t + 2 \int_0^R \!\!\| f({S - t},\cdot) - f({R - t},\cdot) \|^2_{L^2}
  \d  t\\
  & = & 2 \int_0^{S - R} \| f({t},\cdot) \|^2_{L^2} \d  t 
  + 2 \int_0^R \| f({S - R + t},\cdot) - f({t},\cdot) \|^2_{L^2} \d  t\\
  & \leqslant& 2 \|f\|_{\mathcal{K}_\gamma}^2\cdot|S-R|^{2\gamma} .  
\end{eqnarray*}
\end{proof}
The following Lemma states that the bounds previously obtained for fixed $L>0$
actually yield a bound in $C^0_{2\gamma,T}$.
\begin{lemma}
\label{lem:C0infty}
Fix $T>0$ and $\gamma>0$ and let $u$ be a random variable such that 
 for all $p>0$ there is a constant $C>0$ such that 
 \[
 \E\|u\|^p_{C^0([0,T]\times [-L,L])} \leqslant C\left(L^\gamma \lambda\right)^p\;,
 \]
 then for $p>\sfrac{1}{\gamma}$
 \[
 \P(\|u\|_{\cC^0_{2\gamma,T}} \geqslant K ) \leqslant C \sum_{L\in\N} L^{ -\gamma p} \Big(\frac\lambda{K}\Big)^p\;.
  \]
\end{lemma}
\begin{proof}
Using Chebychev inequality yields
 \begin{align*}
  \P(\|u\|_{\cC^0_{2\gamma,T}} \geqslant K )
  & =   \P\left( \exists L\in\N \ :\ L^{-2\gamma} \|u\|_{C^0([0,T]\times [-L,L])} \geqslant K \right) \\
  &\leqslant \sum_{L\in \N} \P\left(   \|u\|_{C^0([0,T]\times [-L,L])} \geqslant KL^{2\gamma} \right) \\  
  &\leqslant \sum_{L\in \N}   C   \lambda^p   ( KL^\gamma)^{-p} \;.
 \end{align*}
\end{proof}


 \subsection{Estimates for Gaussian initial conditions}


 In this section, we focus on technical results necessary 
 to treat the term $e^{t\cL_\varepsilon}u_0$ from the mild formulation,
 where the initial condition $u_0(x)=A(\varepsilon x)e^{ix}+c.c.$ is given by a modulated wave 
 with centred Gaussian amplitude $A$. Due to Gaussianity, 
 we need much less regularity of $A$ than we would need in the deterministic case.
 \begin{lemma}
  \label{lem:GausIC}
 Let $A$ be a centred $\C$-valued Gaussian   
 with covariance operator $Q$ given by a Fourier multiplier $q(k) \geq 0$.
 Let $\cH_\tau$ be as in Lemma \ref{lem:sumuplem}.
 
 Define 
 \[
 \|f\|^2_{\cL^\infty_Q} = \sup_{S\in[0,T]} \int_\R  q(k) |f(S,k)|^2 (|k|^{2\gamma}+1)\d k
 +  \sup_{S,R\in[0,T]} \int_\R q(k) \frac{|f(S,k)-f(R,k)|^2}{|S-R|^{2\gamma}} \d k.
 \]
 Then for all $p>0$, $T>0$, and $\kappa>0$ there is a constant $C >0$ such that
 \[
 \E \|\cH_\tau A \|^p_{C^0([0,T]\times[-L,L])} 
 \leqslant    C L^{\gamma p} \|f\|^p_{\cL^\infty_Q}\;.
 \]
 \end{lemma}
 \begin{proof}
  We know that  $A=\sum  n_\ell Q^{1/2}e_\ell$ for any
  orthonormal basis $\{e_\ell\}_{\ell\in\N}$ and standard $\C$-valued Gaussians  $\{n_\ell\}_{\ell\in\N}$, 
  because any Gaussian can be written in
  terms of the covariance operator $Q$ and a cylindrical Gaussian $\sum  n_\ell e_\ell$. 
  Note that the symmetric operator $Q^{1/2}$ has Fourier multiplier $q^{1/2}$.
  
Using again fractional Sobolev spaces with $\gamma>\alpha>2/p$ and  Gaussianity,
it is sufficient to bound three second moments. First
\begin{align*}
 \E |\cH_SA(X)- \cH_SA(Y)|^2 
 & =   \E  \Big|\int_\R [H_S(X-Y)- H_S(X-Y) ] A(Y) \d Y \Big|^2 \\
  & =   \E \Big| \sum_{l\in\N} \int_\R [H_S(X-Y)- H_S(X-Y) ] Q^{1/2} e_\ell (Y) \d Y n_\ell   \Big|^2 \\
    & =    \sum_{l\in\N} \E \Big| \langle [H_S(X-\cdot- H_S(X-\cdot) ] , Q^{1/2} e_\ell \rangle_{L^2} \Big|^2 \\
  & =    \sum_{l\in\N} \Big|\langle Q^{1/2} [H_S(X-\cdot)- H_S(X-\cdot) ]  ,  e_\ell \rangle_{L^2} \Big|^2 \\ 
   & =    \| Q^{1/2} [H_S(X-\cdot)- H_S(X-\cdot) ]  \|^2_{L^2} \\     
     & =    \| q^{1/2} \cF^{-1} [H_S(X-\cdot)- H_S(X-\cdot) ]  \|^2_{L^2} \\ 
      & =   C \int_\R  q(k) |f(S,k)|^2 |k|^{2\gamma}\d k \cdot |X-Y|^{2\gamma}  \;,
\end{align*}
where we again used Parceval, Plancherel and in the final step~\eqref{e:Hdiff}.

Secondly, we can proceed similar to the first case to obtain
\begin{align*}
 \E |\cH_S A(X)- \cH_R A(X)|^2 
 & =  \int_\R q(k) |\cF^{-1}(H_S-H_R)(k)|^2 \d k \\
 & \leqslant \sup_{S,R\in[0,T]} \int_\R q(k) \frac{|f(S,k)-f(R,k)|^2}{|S-R|^{2\gamma}} \d k \cdot |T-R|^{2\gamma}\;.
\end{align*}
 And finally, we can verify
 \begin{align*}
 \E |\cH_TA(X)|^2 =  C \int_\R  q(k) |f(T,k)|^2 \d k.
\end{align*}
 \end{proof}

\section{Main Results}
\label{sec:main}

This section provides the main results, where we focus on the new stochastic estimates 
and treat the deterministic estimates needed as an assumption. 
We give
\begin{itemize}
 \item Approximation Result for Stochastic Convolutions
 \item Attractivity for deterministic initial conditions
 \item Full linear Approximation Theorem
\end{itemize}

Let us first state the general formulation, and fix $W^{(\varepsilon)}$ to be the 
proper rescaled Wiener process as in the rescaling Lemma \ref{lem:res}.
We consider a mild solution of the linear problem
\[ \partial_t u =\mathcal{L}u + \nu \varepsilon^2 u + \varepsilon^{3 / 2}
   \partial_t W^{(\varepsilon)},  \qquad u(0)=u_0\;,
\]
that is a function $u$ which satisfies
\[ u (t) = e^{t\mathcal{L}} u_0 + \nu \varepsilon^2 \int_0^t e^{(t - s)
   \mathcal{L}} u \d  s +\varepsilon^{3 / 2}  W^{(\varepsilon)}_{\mathcal{L}} (t, x) \;.
   \]
In the whole paper, we need the following assumption: 
\begin{equation*}
	|\nu|\leqslant 1,\qquad \textup{i.e.}\qquad \nu =\cO(1).
\end{equation*}
   
   Defining $\cL_\varepsilon= \cL+\nu \varepsilon^2$,
we can rewrite the mild formulation in a slightly different fashion, that turns out to
be easier to use:
\[
u(t) = e^{t\mathcal{L}_\varepsilon}u_0 + \varepsilon^{3 / 2}  W^{(\varepsilon)}_{\mathcal{L}_\varepsilon}(t).
\]
We also have on the slow scale an approximating amplitude equation (AE)
with a $\C$-valued Wiener process and a solution $A$ in the mild form:
\begin{equation*}
\begin{split}
A (T) & = e^{4 T \partial_x^2 } A_0 + \nu \int_0^T e^{(T - S) 4 \partial_x^2}
   A \d  S + {W}_{4 \partial^2_x} (T)\\
   & = e^{4T(\partial_x^2+\nu)}A_0 + {W}_{4 \partial^2_x+\nu} (T).
\end{split}
\end{equation*} 
So now for the first step we assume initial conditions to be $0$
and aim for the result for the stochastic convolutions only.

\subsection{Approximation Result for Stochastic Convolutions}

Our goal in this section is
\begin{equation*}
 \|\varepsilon^{ 1 / 2} W^{(\varepsilon)}_{\mathcal{L}_{\varepsilon}} (t, x) - [ 
   {W}_{4 \partial_x^2+\nu} (t \varepsilon^2, x \varepsilon) e^{ix} +
   \textrm{c.c}.] \|_{C^0_\gamma}
   \quad \text{is small} 
\end{equation*}
for some small $\gamma>0$. Note that by Lemma \ref{lem:C0infty} it is sufficient
to provide bounds on moments of $C^0([0,T] \times[-L,L])$-norms.

Now we can throw   in the definition of $W_{\mathcal{L}^{(\varepsilon)}_{\varepsilon}}$,
and rescale as in  Lemma \ref{lem:res}.
 We obtain for some suitable convolution operator $\mathcal{F}_\tau^{(\varepsilon)}$ with corresponding kernel $F_\tau^{(\varepsilon)}$ 
\begin{multline*}
W^{(\varepsilon)}_{\mathcal{L}_{\varepsilon}} (T \varepsilon^{- 2}, X \varepsilon^{- 1}) = \int_0^T \mathcal{F}^{(\varepsilon)}_{T-s} \d W(s)e^{i X/ \varepsilon}  + c.c.=
\\
    =\frac{\varepsilon^{- 1 / 2}   }{2 \pi}  \sum_{\ell \in \N} \int_0^T
   \int_{\R} \int_{- \frac1\varepsilon}^{\infty} 
   e^{(T - S)[\nu-(2 + k\varepsilon)^2 k^2]} e^{ik (X - Y)} \d  k e_\ell (Y) \d  Y \d 
   \beta_\ell (S) e^{i X/ \varepsilon}  + \textrm{c.c.}
   \end{multline*}
Thus
\begin{multline*}
  \Big|  \varepsilon^{1 / 2} W^{(\varepsilon)}_{\mathcal{L}_{\varepsilon}} (T\varepsilon^{-2}, X\varepsilon^{-1}) - 
   {W}_{4 \partial_x^2+\nu} (T, X) e^{iX/\varepsilon} \Big| \\ 
   =  \Big| \int_0^T [\mathcal{F}_{T-S}^{(\varepsilon)} - e^{4(T-S)[\partial_X^2+\nu]}  \d W(S)  \Big|
   =  \Big| \int_0^T \cH_{T-S}^{(\varepsilon)}  \d W(S)  \Big|\;.
 \end{multline*} 
In view of Lemma \ref{lem:sumuplem} or Lemma \ref{lem:key2}
we define the convolution operator $\cH_{\tau}^{(\varepsilon)}$ with kernel 
$H_{\tau}^{(\varepsilon)}$ as follows
\[
\begin{split}
  H_{\tau}^{(\varepsilon)} (X) & =  \frac{1}{2 \pi}  \int_{- 1 /
  \varepsilon}^{+ \infty} e^{- \tau (2 + k \varepsilon)^2 k^2 + \tau \nu}
  e^{ikX} \d  k - \frac{1}{2 \pi}  \int_{- \infty}^{+ \infty} e^{- \tau 4
  k^2 + \tau \nu} e^{ikX} \d  k\\
  & =  \frac{1}{2 \pi} e^{\tau \nu}  \Big[ \int_{- \infty}^{+ \infty} [e^{-
  \tau (2 + k \varepsilon)^2 k^2} - e^{- \tau 4 k^2}] e^{ikX} \d  k -
  \int_{- \infty}^{- 1 / \varepsilon} e^{- \tau (2 + k \varepsilon)^2 k^2}
  e^{ikX} \d  k \Big].
\end{split}
\]
Now we take the inverse Fourier Transform and we have (in view of Lemma \ref{lem:sumuplem})
\[
 f (\tau, k) =   e^{\tau \nu} [ e^{- \tau (2 + k \varepsilon)^2 k^2}  - e^{- \tau 4 k^2}]
 - e^{\tau \nu} \chi_{(- \infty, - 1 / \varepsilon)} (k)e^{- \tau (2 + k \varepsilon)^2 k^2} 
\]
Given $\delta \in (0, 1)$ and using the symmetries of the integrals, it is sufficient to consider
 $f$, which is  split into four pieces:
\begin{align}
    f(\tau,k)&=e^{\tau \nu}\chi_{(\delta/\varepsilon,\infty)}(k)e^{-\tau(2+k\varepsilon)^2k^2}\label{eq:fa}\\
&\phantom{=}+
e^{\tau \nu}\chi_{(-1/\varepsilon,-\delta/\varepsilon)}(k)e^{-\tau(2+k\varepsilon)^2k^2}\label{eq:fb}\\
&\phantom{=}+
e^{\tau \nu}\chi_{(-\delta/\varepsilon,\delta/\varepsilon)}(k)(e^{-\tau(2+k\varepsilon)^2k^2}-e^{-\tau4k^2})\label{eq:fc}\\
&\phantom{=}-2 e^{\tau \nu}\chi_{(\delta/\varepsilon,\infty)}(k)e^{-\tau4k^2}\label{eq:fd}\\
& = f_a + f_b + f_c + f_d. \nonumber
\end{align}
Here $f_a$, $f_b$, and $f_d$ are the relatively simple terms that turn out to be small due to fast exponential decay. On the other hand, for $f_c$ we cannot take advantage of exponentially small terms, but we need to rely on the difference being small.

We obtain the following  main Theorem on the approximation of stochastic convolutions.

\begin{thm}\label{thm:M}
For all $T>0$, for all $\kappa >0$, for all $p>0$ and all sufficiently small  $\gamma>0$ 
there is a constant $C>0$ such that
\begin{equation*}
\P(\|\varepsilon^{ 1 / 2} W^{(\varepsilon)}_{\mathcal{L}_{\varepsilon}} (t, x) - [ 
   {W}_{4 \partial_x^2+\nu} (t \varepsilon^2, x \varepsilon) e^{ix} +
   \textrm{c.c}.] \|_{C^0_{\gamma,T\varepsilon^{-2}}}\geqslant\varepsilon^{1-\kappa})\leqslant C \varepsilon^p
\end{equation*}
for all $\varepsilon\in (0,\varepsilon_0)$.
\end{thm}
\begin{proof}
Lemma \ref{lem:key2} together with Lemmas  \ref{lem:new1} and \ref{lem:new2} will
provide bounds in $C^0([0,T]\times [-L,L])$.
Then the claim follows by using Lemma \ref{lem:C0infty}.
\end{proof}
Note that as we want to have the result for all small $\gamma>0$,
we state here the result in $C^0_{\gamma}$ instead of $C^0_{2\gamma}$ as in Lemma \ref{lem:C0infty}.


\subsection{Attractivity for deterministic initial conditions}

This section should motivate, why we assume in the full approximation result 
that the initial condition is a modulated wave that is split in a Gaussian with 
bounds on the covariance operator, that would not allow for the existence 
a derivative, and a more regular part in $ H^1_{\ell, u} $.

We fix a time 
\begin{equation*}
 t_\varepsilon = \mathcal{O}(\varepsilon^{-2})
\end{equation*}
and try to approximate $e^{t_{\varepsilon} \mathcal{L}_{\varepsilon}} u_0$ 
by a modulated wave. It turns out that we obtain  a more regular part 
of the amplitude, and a Gaussian part, that is only in $C^0_\gamma$.  

\begin{multline*}
  \| e^{t_{\varepsilon} \mathcal{L}_{\varepsilon}} u_0 + \varepsilon^{1 / 2}
  W_{\mathcal{L}_{\varepsilon}}^{(\varepsilon)} (t_{\varepsilon}) - (A_{\det}
  (\varepsilon x) e^{ix} + A_{\textrm{st}} (\varepsilon x) e^{ix} + \textrm{c.c.})
  \|  \leqslant\\ \leqslant  \| e^{t_{\varepsilon} \mathcal{L}_{\varepsilon}} u_0 -
  (A_{\det} (\varepsilon x) e^{ix} + \textrm{c.c.}) \|_{C_{\gamma}^0}
   + \| \varepsilon^{1 / 2} W_{\mathcal{L}_{\varepsilon}}^{(\varepsilon)}
  (t_{\varepsilon}) - (A_{\textrm{st}} (\varepsilon x) e^{ix} + \textrm{c.c.})
  \|_{C_{\gamma}^0} .
\end{multline*}
For the first (deterministic) term we use the  deterministic
attractivity theorem, which we state just as an assumption.
See  
~\cite{GS96} for a result in $H^1_{\ell, u}$, which, together with the embedding 
$H^1_{\ell, u} \hookrightarrow C_{\gamma}^0$ proved in Subsection~\ref{sec:localH1}, shows that the following assumption is at least true for all $u_0 \in H^1_{\ell, u}$. 

\begin{assu}
Suppose that for the initial condition $u_0$ there is a smooth function $A_{\det}$ such that 
\[
 \| e^{t_{\varepsilon} \mathcal{L}_{\varepsilon}} u_0 -
  (A_{\det} (\varepsilon x) e^{ix} + \textrm{c.c.}) \|_{C_{\gamma}^0}
  \quad\textrm{is small.}
\]
\end{assu}
For the second, the stochastic term, we use our approximation result of  Theorem \ref{thm:M} for the
stochastic convolution to show that this is small. 
Thus we need to define $A_{\textrm{st}}$ as 
\[ A_{\textrm{st}} (X) = W_{4 \partial^2_X} (\varepsilon^2 t_{\varepsilon}, X)
   \sim \mathcal{N} \Big( 0, \int_0^{\varepsilon^2 t_{\varepsilon}} e^{8 s
   \partial_X^2} \d  s \Big) =\mathcal{N} (0, Q), \]
so we have a covariance operator $Q$ with Fourier-symbol $q$ given by  
\[ q ( {k})
= \int_0^{\varepsilon^2 t_{\varepsilon}} e^{- 8 sk^2} \d  s
   = \frac{1 - e^{- 8 k^2 t_{\varepsilon} \varepsilon^2}}{8 k^2} .
   \]
Direct calculation gives the following estimate on $q ({k})$: there exists some $\delta_0>0$ such that
\begin{equation}
  \label{eq:starq} 
  q ({k}) \leqslant C \left\{ \begin{array}{ll}
    1 \hspace*{\fill} \, & \textrm{if }\ | k | \leqslant
    \delta_0 \\
    \frac{1}{k^2} & \text{if }\ | k | \geqslant \delta_0
  \end{array} \right\} 
  = C \min\{ 1,\ k^{-2} \}
\end{equation}
This will be needed as an assumption for the full approximation result.

%
\subsection{Full approximation}

In this section, we will prove the general approximation result for 
initial conditions that are almost a modulated wave. To be more precise, we assume
\begin{assu}
\label{a:ICsplit}
Consider for some sufficiently small $\frac12>\kappa>0$ the splitting
 \[ u_0(x) = A_{\det} (\varepsilon x) e^{ix} + A_{\textrm{st}} (\varepsilon x)
   e^{ix}  + \textrm{c.c.} + \varepsilon^{1-\kappa} E \;,
\]
where we assume for some sufficiently small $\kappa>\gamma>0$
 the following:
\[ \|E\|_{C^0_\gamma} \leqslant C, \qquad   \|A_{\det}\|_{H_{\ell, u}^1} \leqslant C , 
\qquad  \hspace*{\fill} 
A_{\textrm{st}} \sim \mathcal{N}(0, Q) \;,
   \]
where the $Q$ is such that~\eqref{eq:starq} holds.
\end{assu}
Let's define for ease of notation $A_0 = A_{\det} + A_{\textrm{st}}$. 
Again, for ease of notation's  we define the solution $u$ and the approximation $u_A$ as
\[ 
\begin{split}
u (t, \cdot) 
 & = e^{t \mathcal{L}_{\varepsilon}} u_0 +
   \varepsilon^{1 / 2} W_{\mathcal{L}_{\varepsilon}}^{(\varepsilon)} (t)\\
    u_A (t, x) &= [e^{4 t \varepsilon^2 \partial_X^2} A_0]  
   (\varepsilon x)e^{ix} + W_{4 \partial_x^2} (t \varepsilon^2, \varepsilon x) e^{ix} +
     \textrm{c.c.}
\end{split}      
\]
For the deterministic approximation result we use the following Assumption.
For a full deterministic approximation result in the space $H^1_{\ell,u}$ 
see for example~\cite{GS96}.
\begin{assu}
\label{a:detA}
Let's define 
\begin{equation*}
\mathcal{E}_\textrm{det}^{(\varepsilon)}=\sup_{t\in[0,T_0 \varepsilon^{-2}]}\|e^{t\mathcal{L}_{\varepsilon}} [A_{\textrm{det}} (\varepsilon x) e^{ix}] - [e^{4\partial_X^2\varepsilon^2t} A_{\textrm{det}}](\varepsilon x) e^{ix}]\|_{C_{\gamma}^0}.
\end{equation*}
We assume that this is a good approximation for the deterministic part, i.e. 
\begin{equation*}
\mathcal{E}_\textrm{det}^{(\varepsilon)}\to 0 \qquad \textrm{for}\ \varepsilon\to 0.
\end{equation*}
\end{assu}
\begin{thm}[Approximation]
\label{thm:fullmain}
Under Assumptions \ref{a:ICsplit} and \ref{a:detA}, for all $\kappa>0$ and $\gamma \in(0,\kappa)$ both sufficiently small 
and for all $T_0 > 0$ and $p>1$ there is a constant $C>0$ such that 
  \[ \P \Big(\sup_{[0, T_0 \varepsilon^{- 2}]} \| u - u_A
     \|_{C_{\gamma}^0} \leqslant  \mathcal{E}_\textrm{det}^{(\varepsilon)} + \varepsilon^{1/5}\Big) \geqslant 1 - C_p
     \varepsilon^p \]
     for all $\varepsilon\in(0,1)$.
\end{thm}
We do not claim that the bound by $\varepsilon^{1/5}$ is optimal, 
but it seems that substantial improvements will be significantly more technical.
\begin{proof}
Using Assumption \ref{a:ICsplit} there are four terms in $u-u_A$,
which we need to bound. There are three from the splitting 
in $E$,  $A_{\det}$, and  $A_{\textrm{st}}$ and a fourth one  
from the difference of the stochastic convolution.
We proceed in several steps. 
  
First for the term with $E$ we show by Lemma \ref{lem:bSG} for all $t\geqslant 0$ that 
  \[ 
  \| e^{t\mathcal{L}_{\varepsilon}} \varepsilon^{1-\kappa} E  \|_{C_{\gamma}^0} 
  \leqslant C \varepsilon^{1 - \kappa - \gamma/2}  \| E \|_{C_{\gamma}^0} \;.
  \]
Secondly,  we approximate the difference of the two stochastic convolutions using Theorem \ref{thm:M}.
Thus 
\[
\|\varepsilon^{ 1 / 2} W^{(\varepsilon)}_{\mathcal{L}_{\varepsilon}} (\varepsilon^{-2}T, \varepsilon^{-1}T) - [ 
   {W}_{4 \partial_x^2+\nu} (T, X) e^{iX/\varepsilon} +
   \textrm{c.c.}] \|_{C^0_{\gamma,T}}\leqslant\varepsilon^{1-\kappa}
\]
with probability $\cO(\varepsilon^p)$ for all $p>1$.
 
Finally, for the remaining last two terms including $A_0$, 
we first
use Assumption \ref{a:detA} for the term containing $A_{\det}$.

For the final term containing 
$A_{\textrm{st}}$, we 
rewrite 
using a simple rescaling stated in Lemma \ref{lem:rewriting} below
\[ e^{t\mathcal{L}_{\varepsilon}} (A_{\textrm{st}} (\varepsilon x) e^{ix}) 
= [e^{4\partial_X^2\varepsilon^2t} A_{\textrm{st}}](\varepsilon x) e^{ix}
+ \tilde{\mathcal{H}}_{\varepsilon^2 t} A_{\textrm{st}} (\varepsilon x) e^{ix}\;.
\]
Thus it remains to bound 
\[
\sup_{T\in [0,T_0]}\| \tilde{\mathcal{H}}_{T} A_{\textrm{st}} (\varepsilon x)e^{ix}\|_{C^0_\gamma} 
\leq \sup_{[0,T_0]}\| \tilde{\mathcal{H}}_{T} A_{\textrm{st}} (\varepsilon \cdot)\|_{C^0_\gamma}
\leq \sup_{[0,T_0]}\| \tilde{\mathcal{H}}_{T} A_{\textrm{st}} \|_{C^0_\gamma}\;,
\]
where we used Lemma \ref{lem:chsp}. Now we proceed with Lemma \ref{lem:GausIC} and  Lemma \ref{lem:C0infty}
that gives
\[  \E \sup_{[0,T_0]}\| \tilde{\mathcal{H}}_{T} A_{\textrm{st}} \|_{C^0_\gamma}^p
\leq C \|\tilde{f}\|^p_{\cL_Q^\infty}\;,
\]
which is small as argued below in Section \ref{sec:errgauss}. Note that in view of Lemma \ref{lem:C0infty}
we need here the ${\cL_Q^\infty}$-norm
for $\gamma/2$ and not $\gamma$, but we bound it for any $\gamma>0$ later anyway.

\end{proof}
For the terms involving the initial conditions we used the following rescaling lemma.
\begin{lemma}\label{lem:rewriting}
The following holds:
\[
e^{T\varepsilon^-2\cL_\varepsilon}[A(\varepsilon x)e^{ix}]-[e^{(4\partial_X^2+\nu) T}A](\varepsilon x)e^{i x} 
= \tilde{\cH}_T A (\varepsilon x) e^{ix}
\]
with $\tilde{\cH}_T$ having kernel 
\[
\tilde{f}(T,\ell)= e^{-T(2+\varepsilon\ell)^2\ell^2+\nu T} - e^{-4T\ell^2}\;. 
\] 
\end{lemma}

\begin{proof}
We have that
\begin{align*}
 e^{T\varepsilon^-2\cL_\varepsilon}[A(\varepsilon x)e^{ix}] 
&= \int_\R   e^{-T[(1 -k^2)^2-\nu\varepsilon^2]}\frac1\varepsilon\widehat{A}((k-1)/\varepsilon)  \d  k \\
&= \int_\R  e^{-T[\ell^2(2+\ell\varepsilon)^2-\nu]}\widehat{A}(\ell) e^{i\ell\varepsilon x} \d  \ell \cdot e^{ix}.
\end{align*}
Moreover, 
\[
[e^{(4\partial_X^2+\nu) T}A](\varepsilon x) \cdot e^{i x} 
=  \int_\R  e^{-4 T k^2}\widehat{A}(\ell) e^{ik\varepsilon x} \d  \ell \cdot e^{ix}\;.
\]
\end{proof}

\begin{remark}
Note that  $\tilde{\cH}_T=\cH_T+R_T$, where $\cH_T$ has kernel $f$, 
the one we already introduced and studied in detail and an additional 
remainder $R_T$ which has kernel $f-\tilde{f}$. 
In particular this kernel has only parts in the exponential tail, so the error is easily bounded.

The error terms come from the fact that we do not cut in 0,
but go further left to $-1/\varepsilon$. 
Also, we have contributions coming from the complex conjugate, but they do not cancel out.
\end{remark}

\section{Error estimates in space}
\label{sec:errspace}

This section will provide the technical bounds on $f_a$, $f_b$, $f_c$, and $f_d$  
defined in~\eqref{eq:fa}--\eqref{eq:fd} in the $L^2(H^\gamma)$-norm.
This is crucial for applying Lemma~\ref{lem:key2} in the proof of Theorem~\ref{thm:M}.
These are all direct estimates that do not rely on any other result.

Let us  first remark that again we can bound separately the contribution with weight $k^{2\gamma}$ and with 1. 
Moreover, we can preform the computation for $k^{2\gamma}$ only, and then consider the case $\gamma = 0$
 to treat the 1. 

We use the following observations. For $|\nu|\leqslant 1$ and $\varepsilon k\geqslant \delta$ we have
\[
|f_a|\leqslant e^{\tau \nu}e^{-\tau(2+\delta)^2k^2}\leqslant Ce^{-\tau 4 k^2},
\qquad\text{and}\qquad  |f_d| \leqslant C e^{-\tau4k^2}.
\]
Thus we obtain
\begin{equation}\allowdisplaybreaks
  \label{eq:fafd}
  \begin{split}
  \lefteqn{  \int_0^T\int_\R |f_a+f_d|^2 |k|^{2\gamma}\d k \d\tau 
     \leqslant  C\int_0^T \int_{\delta/\varepsilon}^\infty  e^{-8\tau k^2} k^{2\gamma}\d k \d\tau}\\
 &\leqslant C\int_0^T \int_{\delta/\varepsilon}^\infty  k^{2\gamma}e^{-4\tau k^2} \d k\;  e^{-4\tau \delta^2/\varepsilon^2} \d\tau\\
&\leqslant C \int_0^T e^{-4\tau  \delta^2\varepsilon^{-2}}\tau^{-(2\gamma+1)/2}\d\tau
= C\int_0^{T\varepsilon^{-2}} (\varepsilon^2\sigma)^{-(2\gamma+1)/2}e^{-4\sigma \delta^2} \varepsilon^2 \d\sigma\\
&\leqslant C \varepsilon^{1-2\gamma}\int_0^\infty\sigma^{-(2\gamma+1)/2}e^{-\sigma 4 \delta^2} \d\sigma
= C \varepsilon^{1-2\gamma} \delta^{(2\gamma-1)} \;,
    \end{split}
\end{equation}
where, in order to be able to integrate in $\tau$, we must take $-\gamma-1/2>-1$, that is $\gamma < 1/2$.

By the same estimates we can bound the contribution of the term $f_b$.
\begin{equation*}\allowdisplaybreaks
    \int_0^T\int_\R|f_b|^2 |k|^{2\gamma} \d k \d\tau =
    \int_0^T\int_{-1/\varepsilon}^{-\delta/\varepsilon}|f_b|^2
    |k|^{2\gamma} \d k \d\tau
  \leqslant  C \int_0^T \int_{\delta/\varepsilon}^{\infty}e^{-8\tau k^2}
     k^{2\gamma}\d k \d\tau\;.
\end{equation*}
Now we turn to the complicated term $f_c$.
By using the mean value theorem
we derive
\begin{equation*}
  \begin{split}
    e^{-\tau(2+k\varepsilon)^2k^2}-e^{-\tau4k^2} = -\tau
    e^{-\tau\xi}[(2+k\varepsilon)^2-4]k^2\\
= -\tau e^{-\tau\xi}\varepsilon k k^2(4+k\varepsilon)
  \end{split}
\end{equation*}
with $\xi$ taking value in $[4k^2,4k^2(1+k\varepsilon/2)^2]$, with the
additional condition, given by the indicator function, that $k\in
(-\delta/\varepsilon,\delta/\varepsilon)$. So the extremes of the interval
for $\xi$ are actually 
\[
\Big[\Big(1-\frac{\delta}{2}\Big)^2  4 \frac{\delta^2}{\varepsilon^2},4
\frac{\delta^2}{\varepsilon^2}\Big]
\quad\text{or}\quad 
\Big[4 \frac{\delta^2}{\varepsilon^2}, \Big(1+\frac{\delta}{2}\Big)^2  4 \frac{\delta^2}{\varepsilon^2}\Big]
\]
depending on $k$ being negative or positive, respectively. We are
interested in the absolute value of $f$, so we have, as $\tau \in[0,T]$,
\begin{equation}
  \label{eq:boundfcabs}
  \begin{split}
    |f_c| & 
    \leqslant \chi_{(-\delta/\varepsilon,\delta/\varepsilon)}e^{\nu \tau}\varepsilon \tau k^3(4+k\varepsilon) e^{-\tau
      (1-\frac{\delta}{2})^24k^2}\\
 &\leqslant \chi_{(-\delta/\varepsilon,\delta/\varepsilon)} C \varepsilon \tau k^3 e^{-\tau (1-\frac{\delta}{2})^2 4k^2}\\
 &\leqslant \chi_{(-\delta/\varepsilon,\delta/\varepsilon)} C \varepsilon \tau k^3 e^{-\tau 4k^2}\;.
  \end{split}
\end{equation}
Let's now fix $0<\mu<1/2-\gamma$, in order to obtain
\begin{equation}\allowdisplaybreaks
  \label{eq:intboundfc}
  \begin{split}
    \int_0^T\int_\R|f_c|^2|k|^{2\gamma}\d k \d\tau &=
    \int_0^T\int_{-\delta/\varepsilon}^{\delta/\varepsilon}|f_c|^2|k|^{2\gamma}\d k \d\tau\\
&\leqslant C \int_0^T \tau^2 \int_{-\delta/\varepsilon}^{\delta/\varepsilon}\varepsilon^2 |k|^{6+2\gamma}e^{-8\tau
      k^2}\d k \d\tau\\
&\leqslant C \varepsilon^2 \int_0^T \tau^2 \int_{0}^{\delta/\varepsilon}
k^{2-\mu} k^{4+2\gamma+\mu}e^{-8\tau
      k^2}\d k \d\tau\\
&\leqslant C \varepsilon^\mu \int_0^T \tau^2 \int_{0}^{\infty} k^{4+2\gamma+\mu}e^{-8\tau
      k^2}\d k \d\tau\\
&\leqslant C \varepsilon^\mu \int_0^T \tau^{-1/2-\gamma-\mu/2} \d\tau.
  \end{split}
\end{equation}
\begin{remark} In the case $\nu<0$ here and later many terms can be bounded independently of time, 
While for $\nu>0$ our constants usually depends exponentially on $T$.
\end{remark}

So if we now put all contributions together we get the following bound:
\begin{equation*}
\|f\|^2_{L^2(H^\gamma)}\leqslant C[ \varepsilon^{1 - 2 \gamma} +\varepsilon + \varepsilon^\mu],
\end{equation*}
where we need $\gamma < 1 / 2$ and $\mu + 2\gamma <1$.

\begin{lemma}
\label{lem:new1}
For all $T>0$ and all $\kappa>0$, there exist $\varepsilon_0>0$, $\gamma_0>0$, and $C>0$ such that 
\begin{equation*}
\|f\|^2_{L^2(H^\gamma)}\leqslant C \varepsilon^{1-\kappa},
\end{equation*}
for all $\gamma\in (0,\gamma_0)$ and   all $\varepsilon\in (0,\varepsilon_0)$.
\end{lemma}


\section{Error estimates -- 2: Time }
\label{sec:errtime}


In Lemma \ref{lem:key2} we provided a bound in terms of the norm $ \|f\|_{\mathcal{K}_\gamma}^2$, which is defined as
\begin{multline*}
\|f\|_{\mathcal{K}_\gamma}^2
 = \sup_{S\in[0,T]} S^{-2\gamma} \int_0^S \|f(\tau,\cdot)\|^2_{L^2} \d\tau +\\
 +   \sup_{0\leqslant R \leqslant S \leqslant T} (S-R)^{-2\gamma}\int_0^R \|f(S-R+\tau,\cdot)-f(\tau,\cdot)\|^2_{L^2} \d\tau \;.
\end{multline*} 
Now we evaluate that explicitly.

The first term is easily bounded as in Section \ref{sec:errspace}. We obtain the following bound for $\gamma$ is close to $0$:
\begin{equation}\label{e:spaceintime}
\sup_{S\in [0,T]} S^{-2\gamma} \int_0^{S}\|f(\tau,\cdot)\|^2_{L^2} \d\tau \leqslant  C \varepsilon^{1-4\gamma} \;.
\end{equation}
The term $S^{-2\gamma}$ did not appear in the estimates in Section~\ref{sec:errspace}, but is easily controlled. 
We comment on the proof in more detail below.
When we are considering $f_c$ there was already a term in $S$ showing, and we have from~\eqref{eq:intboundfc}:
\begin{equation*}
\sup_{S\in [0,T]} S^{-2\gamma} \varepsilon^\mu \int_0^{S} \tau^{-1/2-\mu/2}\d \tau \leqslant \varepsilon^\mu \sup_{S\in [0,T]} S^{1/2-\mu/2-2\gamma},
\end{equation*}
with $\mu<1$ to guarantee the integrability in 0 and $\mu=1-4\gamma$ to bound the supremum contribution by a constant and have the maximum $\mu$ possible, with $\gamma<\sfrac{1}{4}$.
For the other terms, we follow the estimates in~\eqref{eq:fafd}, and we have
\begin{equation*}
\begin{split}
\sup_{S\in [0,T]} S^{-2\gamma} \int_0^{S}\|f_{a,b,d}(\tau,\cdot)\|^2_{L^2} \d\sigma &\leqslant \sup_{S\in [0,T]} S^{-2\gamma} C \varepsilon \int_0^{S\varepsilon^{-2}}\sigma^{-1/2} e^{-4\delta^2\sigma}\d\sigma\\
& \leqslant \varepsilon^{1-4\gamma} \sup_{R\in [0,+\infty)} R^{-2\gamma}  \int_0^{R}\sigma^{-1/2}e^{-4\delta^2\sigma} \d\sigma.
\end{split}
\end{equation*}
Now 
$
R^{-2\gamma}  \int_0^{R}\sigma^{-1/2}e^{-4\delta^2\sigma} \d\sigma
$
can be bounded with $R^{-2\gamma}$ for $R>1$ and with $R^{1/2-2\gamma}$ for $0\leqslant R \leqslant 1$, and putting everything together we have~\eqref{e:spaceintime}.

We can now move on to the second term.
Using first that $f$ is bounded by a constant and then the mean value theorem, we have that for some $\xi$ between 0 and $S-R$ and $\eta \in (0,1)$, 
the second term is bounded by
\begin{equation*}
\begin{split}
\lefteqn{\sup_{0 \leqslant R \leqslant S \leqslant T} (S - R)^{- 2
  \gamma}  \int_0^R \| f (S - R + \tau, \cdot) - f (\tau, \cdot)
  \|^2_{L^2} \d  \tau}\\  
  & \leqslant  \sup_{0 \leqslant R \leqslant S \leqslant T} (S - R)^{\eta -
  2 \gamma} C_{\eta}  \int_0^R \int_{\R} \left|
  \frac{\partial}{\partial t}  (f_a + f_d + f_b + f_c) |_{t = \tau + \xi}
   \right|^{\eta} \d  k \d  \tau\\
  [\textrm{different} \chi] & =  \sup_{0 \leqslant R \leqslant S \leqslant T}
  (S - R)^{\eta - 2 \gamma} C_{\eta}  \int_0^R \int_{\R} \left|
  \frac{\partial}{\partial t}  (f_a + f_d) |_{t = \tau + \xi}  
  \right|^{\eta} \d  k \d  \tau\qquad \textrm{(A+D)}\\
  &   + \sup_{0 \leqslant R \leqslant S \leqslant T} (S - R)^{\eta - 2
  \gamma} C_{\eta}  \int_0^R \int_{\R} \left|
  \frac{\partial}{\partial t}  (f_b) |_{t = \tau + \xi}  
  \right|^{\eta} \d  k \d  \tau\qquad \qquad\textrm{(B)}\\
  &   + \sup_{0 \leqslant R \leqslant S \leqslant T} (S - R)^{\eta - 2
  \gamma} C_{\eta}  \int_0^R \int_{\R} \left|
  \frac{\partial}{\partial t}  (f_c) |_{t = \tau + \xi}  
  \right|^{\eta} \d  k \d  \tau \;. \qquad\qquad\textrm{(C)}
\end{split}
\end{equation*}
We consider the three components separately.

We start with B, and the same ideas will provide the bounds also for A and D.
\begin{equation*}
\begin{split}
  (B) & =  \sup_{0 \leqslant R \leqslant S \leqslant T} (S - R)^{\eta - 2
  \gamma} C_{\eta}  \int_0^R \int_{\R} \left|
  \frac{\partial}{\partial t}  (f_b) |_{t = \tau + \xi}  
  \right|^{\eta} \d  k \d  \tau\\
  & =  \sup_{0 \leqslant R \leqslant S \leqslant T} (S - R)^{\eta - 2
  \gamma} C_{\eta}  \int_0^R \int_{- 1 / \varepsilon}^{- \delta / \varepsilon}
  e^{\eta (\tau + \xi) \nu}  | (\nu  - (2 + k \varepsilon)^2 k^2) e^{- (\tau + \xi)  (2 + k \varepsilon)^2
  k^2} |^{\eta} \d  k \d  \tau\\
  & \leqslant  C \sup_{0 \leqslant R \leqslant S \leqslant T} (S - R)^{\eta -
  2 \gamma}  \int_0^R  \int_{\delta /
  \varepsilon}^{1 / \varepsilon} (1 + k^{2 \eta}) e^{- \eta
  8\tau k^2} \d  k \d  \tau\\ 
  & =  C \sup_{0 \leqslant R \leqslant S \leqslant T} (S - R)^{\eta - 2
  \gamma} \int_0^R  \int_{\delta /
  \varepsilon}^{1 / \varepsilon} (1 + k^{2 \eta}) e^{- \eta 4
  \tau k^2} \d  ke^{- \eta 4 \tau \delta^2 /
  \varepsilon^2} \d  \tau\\
  & \leqslant   C \sup_{0 \leqslant R \leqslant S \leqslant T} (S - R)^{\eta - 2
  \gamma}  \int_0^R \tau^{- (2 \eta + 1) / 2} e^{- \eta 4 \tau
\delta^2 / \varepsilon^2} \d  \tau\\
  & \leqslant  C \sup_{0 \leqslant R \leqslant S \leqslant T} (S - R)^{\eta -
  2 \gamma} (\varepsilon^{1 - 2 \eta} + \varepsilon) 
\end{split}
\end{equation*}
We need $\eta\geqslant 2\gamma$, so we take $\eta-2\gamma$, and for the integrability in time we need $\gamma<\sfrac{1}{4}$.

The pieces (A) and (D), as anticipated, are bounded exactly in the same way, up to different constants, so we skip the details.

We are just left with the (C) part to estimate. Our best option to get rid of the
difference in this case is to use the mean value theorem a second time, in $k$. We obtain
\begin{equation}\label{eq:estcsec6}
\begin{split}
  \textrm{(C)} & = C \sup_{0 \leqslant R \leqslant S \leqslant T} (S - R)^{\eta - 2
  \gamma} \int_0^R \int_{- \delta / \varepsilon}^{\delta /
  \varepsilon} \left| \frac{\partial}{\partial t}  (f_c) |_{t = \tau + \xi}
    \right|^{\eta} \d  k \d  \tau\\
  & \leqslant  C \sup_{0 \leqslant R \leqslant S \leqslant T} (S - R)^{\eta -
  2 \gamma}   \int_0^R e^{\eta (\tau + \xi) \nu}  \int_{- \delta /
  \varepsilon}^{\delta / \varepsilon} \left| \nu (e^{- (\tau + \xi)  (2 + k
  \varepsilon)^2 k^2} - e^{- (\tau + \xi) 4 k^2}) +\right.\\
  & \hspace*{2cm} +\left. \left(- (2 + k \varepsilon)^2
  k^2 e^{- (\tau + \xi)  (2 + k \varepsilon)^2 k^2} + 4 k^2 e^{- (\tau + \xi)
  4 k^2}\right)\right |^{\eta} \d  k \d  \tau\\
  & \leqslant C \sup_{0 \leqslant R \leqslant S \leqslant T} (S - R)^{\eta -
  2 \gamma} \int_0^R   \int_{- \delta /
  \varepsilon}^{\delta / \varepsilon}  | (e^{- (\tau + \xi)  (2
  + k \varepsilon)^2 k^2} - e^{- (\tau + \xi) 4 k^2}) |^{\eta} +\\
  &\hspace*{2cm} + \left| 4 k^2 e^{-
  (\tau + \xi) 4 k^2} - (2 + k \varepsilon)^2 k^2 e^{- (\tau + \xi)  (2 + k
  \varepsilon)^2 k^2} \right|^{\eta} \d  k \d  \tau.
\end{split}
\end{equation}
As anticipated we have now two more instances of the Mean Value Theorem, for
the functions $e^{- tx}$ and $xe^{- tx}$ in the variable $x$, taking values in
the interval 
\[
I:=\left[ 4 k^2, 4 k^2  \Big( 1 + \frac{k \varepsilon}{2}
\Big)^2 \right]\;,
\]
where the extrema of the interval might be switched due to
the (additional) conditions on $k$. But we can bound the size of the interval
anyway (as already done previously in the space bounds):
\[ | I | = \left| 4 k^2  \Big( 1 - \Big( 1 + \frac{k \varepsilon}{2}
   \Big)^2 \Big) \right| = | k |^3 \varepsilon | 4 + k \varepsilon |
   \leqslant 5 \varepsilon | k |^3 . \]
We can now consider the two pieces of (C) (called (C1) and (C2)) separately, one for
each difference and write:

\begin{equation*}
\begin{split}
  \text{(C1)} & \leqslant  C \sup_{0 \leqslant R \leqslant S \leqslant T} (S -
  R)^{\eta - 2 \gamma} \varepsilon^{\eta}  \int_0^R
  (\tau + \xi)^{\eta}   \int_{- \delta /
  \varepsilon}^{\delta / \varepsilon} | k |^{3 \eta} e^{- (\tau + \xi) \rho
  \eta} \d  k \d  \tau\\
  & \leqslant C \sup_{0 \leqslant R \leqslant S \leqslant T} (S - R)^{\eta -
  2 \gamma}  \varepsilon^{\eta} \int_0^R  \int_0^{\delta / \varepsilon} k^{3
  \eta} e^{- \tau k^2  \eta}
  \d  k \d  \tau \;,
\end{split}
\end{equation*}
where we took the value for
$\rho = 4 k^2  \left( 1 -\frac{\delta}{2} \right)^2$ 
that would maximise the exponential,
recalling that $\rho \in \left[ 4 k^2, 4 k^2  \left( 1 + \frac{k
\varepsilon}{2} \right)^2 \right]$. 
Moreover $0<\xi<S-R$. 

Now
\begin{equation*}
\begin{split}
  \text{(C1)} & \leqslant  C \sup_{0 \leqslant R \leqslant S \leqslant T} (S -
  R)^{\eta - 2 \gamma} \varepsilon^{\eta - \eta +
  \mu}  \int_0^R  \int_0^{\delta
  / \varepsilon} k^{2 \eta + \mu} e^{- \tau  k^2 \eta} \d  k \d 
  \tau\\
  & \leqslant C \sup_{0 \leqslant R \leqslant S \leqslant T} (S - R)^{\eta -
  2 \gamma} \varepsilon^{\mu}  \int_0^R \tau ^{- (\mu + 1) / 2-\eta} \d  \tau\\
 & \leqslant C \sup_{0 \leqslant R \leqslant S \leqslant T} (S - R)^{\eta -
  2 \gamma} \varepsilon^{\mu} 
  \leqslant C \varepsilon^{\mu}
\end{split}
\end{equation*}
if $-\mu/2-\eta>-1/2$, i.e. $\mu<1-2\eta$, and we consider the optimal case, $\eta= 2\gamma$, 
so we can take $\mu <1-4\gamma$. This is  slightly less then in the previous cases.

Now we do the same with the term (C2):
\begin{equation*}
\begin{split}
  \text{(C2)} & \leqslant C \sup_{0 \leqslant R \leqslant S \leqslant T} (S -
  R)^{\eta - 2 \gamma} \int_0^R  \int_{-
  \delta / \varepsilon}^{\delta / \varepsilon} \varepsilon^{\eta}  | k |^{3
  \eta} e^{- (\tau + \xi) \eta \rho}  | (1 - (\tau + \xi) \rho) |^{\eta}
  \d  k \d  \tau\\
  & \leqslant C \sup_{0 \leqslant R \leqslant S \leqslant T} (S - R)^{\eta -
  2 \gamma}  \varepsilon^{\eta}  \int_0^R  
  \int_0^{\delta / \varepsilon} (1 + k^{2 \eta}) k^{3
  \eta} e^{- \eta \tau  k^2} \d  k \d  \tau\\
  & \leqslant  C \sup_{0 \leqslant R \leqslant S \leqslant T} (S - R)^{\eta -
  2 \gamma} \varepsilon^{\mu}  \int_0^R  
  \int_0^{\delta / \varepsilon} k^{2 \eta + \mu} e^{- \eta \tau  k^2 }
  \d  k \d  \tau +\\
  &   + C \sup_{0 \leqslant R \leqslant S \leqslant T} (S - R)^{\eta - 2
  \gamma} \varepsilon^{\mu}  \int_0^R   \int_0^{\delta / \varepsilon} k^{5 \eta + \mu} e^{-\eta \tau k^2 } \d  k \d  \tau\;.
\end{split}
\end{equation*}
The first term is exactly what we had for (C1), so we go on only with the second
(C2.2)
\[ \text{(C2.2)} \leqslant C \sup_{0 \leqslant R \leqslant S \leqslant T} (S -
   R)^{\eta - 2 \gamma}  \varepsilon^{\mu}  \int_0^R
   \tau^{- (5 \eta + \mu + 1) / 2} \d  
   \tau \]
which is also very much alike (C1), except from a slightly different exponent.
In this case 
we need $-2<-5\eta -\mu-1)$, i.e. $\mu <1-5\eta$, and by taking $\eta=2\gamma$ as before, we get $\mu < 1-10\gamma$, with $\gamma$ small.
 So we have a final result analogous to the one for the first term~\eqref{e:spaceintime}, namely
\begin{equation}\label{e:spaceintime2}
\sup_{0\leqslant R \leqslant S \leqslant T} (S-R)^{-2\gamma}\int_0^R \|f(S-R+\tau,\cdot)-f(\tau,\cdot)\|^2_{L^2} \d\tau \leqslant c \varepsilon^{1-10\gamma-\kappa}.
\end{equation} 
Putting together~\eqref{e:spaceintime} and~\eqref{e:spaceintime2} we obtain the following bound.
\begin{lemma}
\label{lem:new2}
For all $T>0$ and for all $\kappa>0$, there exist $\varepsilon_0>0$, $\gamma_0>0$, and $C>0$ such that 
\begin{equation*}
\|f\|^2_{\mathcal{K}_\gamma}\leqslant C \varepsilon^{1-\kappa},
\end{equation*}
for all $\gamma\in (0,\gamma_0)$ and for all $\varepsilon\in (0,\varepsilon_0)$.
\end{lemma}
%
\section{Gaussian estimates}\label{sec:errgauss}
%
%
In this section we provide the technical estimates that we need to apply Lemma~\ref{lem:GausIC} in the proof of Theorem~\ref{thm:fullmain}, namely that the terms in 
\[
 \|f\|^2_{\cL^\infty_Q} = \sup_{S\in[0,T]} \int_\R  q(k) |f(S,k)|^2 (|k|^{2\gamma}+1)\d k
 +  \sup_{S,R\in[0,T]} \int_\R q(k) \frac{|f(S,k)-f(R,k)|^2}{|S-R|^{2\gamma}} \d k
\]
are small. We will consider the two suprema separately.

Recall that we have a bound for $q(k)$ from~\eqref{eq:starq}:
$
q(k)\leqslant C \min\{1,k^{-2}\}.
$

\subsection{First supremum}
We want to show for $\gamma\in[0,\sfrac{1}{4})$ that
\[
\sup_{S\in[0,T]} \int_\R  q(k) |f(S,k)|^2 (|k|^{2\gamma}+1)\d  k
\leqslant C(\varepsilon^{2}+\varepsilon^{\sfrac{1}{2}})\;.
\]
To prove it we use the form of $f$, that we know from equations~\eqref{eq:fa}, \eqref{eq:fb}, \eqref{eq:fc} and~\eqref{eq:fd}. 
As it was the case in Sections~\ref{sec:errspace} and~\ref{sec:errtime}, 
we can bound the pieces $f_a$, $f_b$ and $f_d$ in the same way, 
and use a slightly different approach for $f_c$. 
As in previous sections, we can consider separately the terms with $k^{2\gamma}$ and $1$, 
the second being a special case of the first one, when $\gamma =0$.

In the first three cases we have:
\begin{equation*}
\begin{split}
\textrm{(A,B,D)} &\leqslant \sup_{S\in[0,T]} C \int_{\delta/\varepsilon}^{\infty}  q(k) |e^{-4Sk^2}|^2 k^{2\gamma}\d  k\\
&\leqslant \sup_{S\in[0,T]} C \int_{\delta/\varepsilon}^{\infty} k^{2\gamma-2}\d k
\leqslant C  \varepsilon^{1-2\gamma},
\end{split}
\end{equation*}
which is small, as long as $\gamma\in[0,\sfrac{1}{2})$. The remaining new terms from $\tilde{f} - f$ are treated in a similar way.

Finally it remains to treat $f_c$, which requires some more care to treat,
as the exponentials in the integrand cannot be bounded with a constant
as it would result in a diverging integral. 
What we can do is to use the bound \eqref{eq:starq} on $q(k)$ 
and truncate in $\delta_0$ to get rid of the singularity in 0.

To be more precise using~\eqref{eq:boundfcabs} we obtain
\[
\begin{split}
 \int_{\R} q(k) |k|^{2\gamma} |f_c(S,k)|^2 \d k 
 &\leqslant  \int_{-\delta/\varepsilon}^{\delta/\varepsilon} q(k)  \varepsilon^2 S^2 |k|^{6+2\gamma}e^{-8S k^4} \d k \\
  &\leqslant C  \varepsilon^2 S^2 \int_{0}^{\delta_0 } k^{6+2\gamma} \d k 
  +  C \varepsilon^2 S^2 \int_{ \delta_0 }^{\delta/\varepsilon}    
  k^{4+2\gamma} \d k \\
    &\leqslant C \varepsilon^2 S^2 + C \varepsilon^\mu S^2 \int_{ \delta_0 }^{\delta/\varepsilon}    k^{2+2\gamma+\mu} \d k \\
    &\leqslant C \varepsilon^2 S^2 + C \varepsilon^\mu S^{(1-2\gamma-\mu)/2}
\end{split}
\]
provided that $0\leqslant 2\gamma +\mu<1$. We can take $\mu=\sfrac{1}{2}$ and we get for $\gamma\in[0,\sfrac{1}{4})$
\[
 \int_{\R} q(k) |k|^{2\gamma} |f_c(S,k)|^2 \d k 
 \leqslant C(\varepsilon^2 + \varepsilon^{\sfrac{1}{2}})\;.
\]
Note that, in particular, the bound holds also when $\gamma=0$.
%
%
\subsection{Second supremum}
%
%
Also for the second supremum we consider two cases, as above, when splitting $f$. 
We want to prove that for $\gamma\in[0,\sfrac{1}{4})$ one has
\[
\sup_{S,R\in[0,T]} \int_\R q(k) \frac{|f(S,k)-f(R,k)|^2}{|S-R|^{2\gamma}} \d k \leqslant C (\varepsilon^{1-4\gamma}+\varepsilon^{\sfrac{1}{5}})\;.
\]
We will use some of the estimates introduced in Section~\ref{sec:errtime}. 
For the easy pieces, namely $f_a$, $f_b$ and $f_d$, we have
\begin{equation*}
\begin{split}
\textrm{(A+D,B)} &\leqslant \sup_{S,R\in[0,T]} C |S-R|^{-2\gamma} \int_{\delta/\varepsilon}^{\infty} q(k) |f(S,k)-f(R,k)|^2 \d k\\
&\leqslant \sup_{S,R\in[0,T]} C |S-R|^{\eta-2\gamma} \int_{\delta/\varepsilon}^{\infty} k^{-2} |\frac{\partial}{\partial t}f(t,k)|^\eta \d k\\
&\leqslant \sup_{S,R\in[0,T]} C |S-R|^{\eta-2\gamma} \int_{\delta/\varepsilon}^{\infty} k^{2\eta-2} \d k\\
&\leqslant \sup_{S,R\in[0,T]} C |S-R|^{\eta-2\gamma} \varepsilon^{1-2\eta}, 
\end{split}
\end{equation*}
which is small if $\eta<1/2$ and if $\eta\geqslant 2\gamma$, so 
we can take $\eta =2\gamma$.
A similar estimate holds for the additional terms coming from $\tilde{f}-f$, since they are in the symmetric exponential tail.

There remains now just one case to check, what in Section~\ref{sec:errtime} was the term (C); we proceed in a way analogous to~\eqref{eq:estcsec6}.
\begin{equation*}
\begin{split}
\textrm{(C)} &\leqslant C \sup_{S,R\in[0,T]} |S-R|^{-2\gamma} \int_{-\delta/\varepsilon}^{\delta/\varepsilon} q(k)  |f(S,k)-f(R,k)|^2 \d k\\
&\leqslant C \sup_{S,R\in[0,T]} |S-R|^{\eta-2\gamma} \int_{-\delta/\varepsilon}^{\delta/\varepsilon} q(k) |\frac{\partial}{\partial t}f(t,k)|^\eta \d k\\
&\leqslant  C \sup_{S,R\in[0,T]} |S-R|^{\eta-2\gamma} \varepsilon^\eta \int_0^{\delta/\varepsilon}q(k)k^{3\eta}(1+k^{2\eta})\d k\\
[\textrm{Assume }\eta\geqslant 2\gamma]\quad & \leqslant C \varepsilon^\eta\int_0^{\delta/\varepsilon} q(k)k^{3\eta}(1+k^{2\eta})\d k\\
&\leqslant C \varepsilon^\eta[\int_0^{\delta_0}k^{3\eta}(1+k^{2\eta})\d k + \int_{\delta_0}^{\delta/\varepsilon} k^{5\eta-2}\d k]\\
&\leqslant C \varepsilon^\eta(1+[k^{5\eta-1}]_{\delta_0}^{\delta/\varepsilon}])\\
[\textrm{choose }\eta=1/5]\quad & \leqslant C\varepsilon^{1/5}.
\end{split}
\end{equation*}

\begin{remark}
The choice of $\eta = \sfrac{1}{5}$ might be optimised by choosing $\delta\ll 1$.
\end{remark}

\section*{Acknowledgements}
The authors would like to thank D.~Koshnevisan for the proof of the statement in  Remark~\ref{rem:conjlog}.


\begin{thebibliography}{10}

\bibitem{Blo:05}
D.~Bl{\"o}mker.
\newblock Approximation of the stochastic {R}ayleigh--{B}enard problem near the
  onset of convection and related problems.
\newblock {\em Stochastics and Dynamics}, 5(03):441--474, 2005.

\bibitem{Blo05}
D.~{Bl\"omker}.
\newblock {Nonhomogeneous noise and $Q$-Wiener processes on bounded domains.}
\newblock {\em {Stochastic Anal. Appl.}}, 23(2):255--273, 2005.

\bibitem{BlHaPa07}
D.~Bl{\"o}mker, M.~Hairer, and G.~Pavliotis.
\newblock Modulation equations: Stochastic bifurcation in large domains.
\newblock {\em Communications in mathematical physics}, 258(2):479--512, 2005.

\bibitem{DBMPS:01}
D.~Bl{\"o}mker, S.~Maier-Paape, and G.~Schneider.
\newblock The stochastic {L}andau equation as an amplitude equation.
\newblock {\em Discrete Contin. Dyn. Syst. Ser. B}, 1(4):527--541, 2001.

\bibitem{BreLi2006}
Z.~Brze{\'z}niak and Y.~Li.
\newblock Asymptotic compactness and absorbing sets for 2d stochastic
  navier-stokes equations on some unbounded domains.
\newblock {\em Transactions of the American Mathematical Society},
  358(12):5587--5629, 2006.

  \bibitem{CLW:15a}
M. D. Chekroun, H. Liu and S. Wang,
Approximation of Stochastic Invariant Manifolds: Stochastic Manifolds for Nonlinear SPDEs I, 
{\em Springer Briefs in Mathematics}, Springer, New York, 2015.
 
\bibitem{CLW:15b}
M. D. Chekroun, H. Liu and S. Wang,
Stochastic Parameterizing Manifolds and Non-Markovian Reduced Equations: Stochastic Manifolds for Nonlinear SPDEs II,
{\em Springer Briefs in Mathematics}, Springer, New York, 2015.
  
 \bibitem{CoEc:90}
P. Collet and J.-P. Eckmann,
The time dependent amplitude equation for the Swift-Hohenberg problem,
{\em Comm. Math. Phys.}, 132:(1):139--153, 1990.

\bibitem{CoEc:02}
P.~Collet and J.-P.~Eckmann.
\newblock A rigorous upper bound on the propagation speed for the
  {S}wift--{H}ohenberg and related equations.
\newblock {\em Journal of statistical physics}, 108(5):1107--1124, 2002.

\bibitem{CrHo93}
M.~C.~Cross and P.~C.~Hohenberg.
\newblock Pattern formation outside of equilibrium.
\newblock {\em Reviews of modern physics}, 65(3):851, 1993.

\bibitem{DPZa:14}
G.~Da~Prato and J.~Zabczyk.
\newblock {\em Stochastic equations in infinite dimensions}, volume 152.
\newblock Cambridge university press, 2014.

\bibitem{Dal09}
R.~C.~Dalang.
\newblock The stochastic wave equation.
\newblock In {\em A minicourse on stochastic partial differential equations},
  pages 39--71. Springer, 2009.

\bibitem{DaQS11}
R.~C.~Dalang and L.~Quer-Sardanyons.
\newblock Stochastic integrals for spde’s: a comparison.
\newblock {\em Expositiones Mathematicae}, 29(1):67--109, 2011.

\bibitem{DaSS09}
R.~C.~Dalang and M.~Sanz-Sol{\'e}.
\newblock {H}{\"o}lder-{S}obolev regularity of the solution to the stochastic
  wave equation in dimension three.
\newblock {\em Mem. Amer. Math. Soc.}, 199(931):vi+70, 2009.

\bibitem{DLS:03}
J. Duan, K. Lu and B. Schmalfu{\ss}.
Invariant manifolds for stochastic partial differential equations, 
{\em The Annals of Probability}, 31(4):2109--2135, 2003.

\bibitem{DuWa:14} J. Duan and W. Wang,
Effective Dynamics of Stochastic Partial Differential Equations,
 Elsevier, 2014.
 
\bibitem{EckHai2001}
J.-P.~{Eckmann} and M.~{Hairer}.
\newblock {Invariant measures for stochastic partial differential equations in
  unbounded domains.}
\newblock {\em {Nonlinearity}}, 14(1):133--151, 2001.

\bibitem{FI15}
O. Faugeras and J. Inglis,
Stochastic neural field equations: a rigorous footing,
{\em Journal of Mathematical Biology}, 71(2):259--300, 2015.

\bibitem{Fun1995}
T.~{Funaki}.
\newblock {The scaling limit for a stochastic PDE and the separation of
  phases.}
\newblock {\em {Probab. Theory Relat. Fields}}, 102(2):221--288, 1995.

\bibitem{HaLa:15a}
M. Hairer and C. Labb\'e,
Multiplicative stochastic heat equations on the whole space,
Preprint, 2015. \url{http://www.hairer.org/papers/mSHE.pdf}

\bibitem{HaLa:15b}
M. Hairer and C. Labb\'e,
A simple construction of the continuum parabolic Anderson model on $R^2$,
Preprint, 2015. \url{http://www.hairer.org/papers/PAM_R2.pdf}

\bibitem{HoSw92}
P.~Hohenberg and J.~Swift.
\newblock Effects of additive noise at the onset of {R}ayleigh--{B}{\'e}nard
  convection.
\newblock {\em Physical Review A}, 46(8):4773, 1992.

\bibitem{Hutt3}
A.~Hutt.
\newblock Additive noise may change the stability of nonlinear systems.
\newblock {\em EPL (Europhysics Letters)}, 84(3):34003, 2008.

\bibitem{Hutt1}
A.~Hutt, A.~Longtin, and L.~Schimansky-Geier.
\newblock Additive global noise delays {T}uring bifurcations.
\newblock {\em Physical Review Letters}, 98(23):230601, 2007.

\bibitem{Kho14}
D.~Khoshnevisan.
\newblock {\em Analysis of stochastic partial differential equations}, volume
  119.
\newblock American Mathematical Soc., 2014.

\bibitem{KMS92}
P.~Kirrmann, G.~Schneider, and A.~Mielke.
\newblock The validity of modulation equations for extended systems with cubic
  nonlinearities.
\newblock {\em Proceedings of the Royal Society of Edinburgh: Section A
  Mathematics}, 122(1-2):85--91, 1992.

\bibitem{KruSta2014}
J.~{Kr\"uger} and W.~{Stannat}.
\newblock {Front propagation in stochastic neural fields: a rigorous
  mathematical framework.}
\newblock {\em {SIAM J. Appl. Dyn. Syst.}}, 13(3):1293--1310, 2014.

\bibitem{La2015} E. Lang,
A Multiscale Analysis of Traveling Waves in Stochastic Neural Fields.
Preprint, ArXive, 2015. 

\bibitem{MiSc:95}
A. Mielke and G.  Schneider,
Attractors for modulation equations on unbounded domains-existence and comparison,
{\em Nonlinearity}, 8(5):743–768, 1995.

\bibitem{MSZ00}
A.~Mielke, G.~Schneider, and A.~Ziegra.
\newblock Comparison of inertial manifolds and application to modulated
  systems.
\newblock {\em Mathematische Nachrichten}, 214(1):53--69, 2000.

\bibitem{MoBlKl:13}
W.~W. Mohammed, D.~Bl{\"o}mker, and K.~Klepel.
\newblock Modulation equation for stochastic {S}wift--{H}ohenberg equation.
\newblock {\em SIAM Journal on Mathematical Analysis}, 45(1):14--30, 2013.

\bibitem{MZZ:08}
S. Mohammed, T. Zhang and  H. Zhao,
The Stable Manifold Theorem for Semi-linear Stochastic Evolution Equations and Stochastic Partial Differential Equations,
{\em Memoirs of the American Mathematical Society}, 196:1--105, 2008.

\bibitem{Pazy}
A.~Pazy.
\newblock {\em Semigroups of linear operators and applications to partial
  differential equations}, volume~44.
\newblock Springer Science \& Business Media, 2012.

\bibitem{quawat71}
C.~Qualls and H.~Watanabe.
\newblock An asymptotic 0-1 behavior of Gaussian processes.
\newblock {\em The Annals of Mathematical Statistics}, 42(6):2029--2035,
  1971.

\bibitem{RuSi96}
T.~Runst and W.~Sickel.
\newblock {\em {S}obolev spaces of fractional order, {N}emytskij operators, and
  nonlinear partial differential equations}, volume~3.
\newblock Walter de Gruyter, 1996.

\bibitem{GS96}
G.~Schneider.
\newblock The validity of generalized {G}inzburg--{L}andau equations.
\newblock {\em Mathematical methods in the applied sciences}, 19(9):717--736,
  1996.

\bibitem{Wal86}
J.~B.~Walsh.
\newblock {\em An introduction to stochastic partial differential equations}.
\newblock Springer, 1986.

\end{thebibliography}

\end{document}